\documentclass[a4paper,12pt]{amsart}
\usepackage{amsmath}
\usepackage{amsfonts}
\usepackage{amssymb}
\usepackage[english]{babel}
\usepackage{color}
\usepackage{bbding}
\usepackage{esint}
\usepackage{bm}
\usepackage{mathrsfs}
\usepackage[latin1]{inputenc}
\usepackage[english]{babel}
\usepackage{amsthm}
\usepackage{graphicx}
\numberwithin{equation}{section}

\newcommand{\texorpdfstring}[2]{#1}

\newcommand{\de}{{\delta}}
\renewcommand{\H}{\mathbb H}

\newcommand{\ep}{{\epsilon}}

\newcommand{\la}{{\lambda}}
\newcommand{\G}{{\mathbb G}}
\newcommand{\f}{{\phi}}

\newcommand{\R}{\mathbb R}
\newcommand{\cH}{\mathcal H}
\newcommand{\res}{\mathop{\hbox{\vrule height 7pt width .5pt depth 0pt
\vrule height .5pt width 6pt depth 0pt}}\nolimits}
\newcommand{\B}{\mathbb B}
\newcommand{\N}{\mathbb N}

\newcommand{\Shaus}{{\mathcal S}}
\newcommand{\Haus}{\mathcal H}

\renewcommand{\leq}{\leqslant}
\renewcommand{\geq}{\geqslant}

\newcommand{\sgn}{\mathrm{sgn}}
\newcommand{\al}{\alpha}
\newcommand{\be}{\beta}

\newcommand{\cG}{{\mathcal G}}

\newcommand{\ra}{\rightarrow}

\makeatletter
\def\@eqnnum{\hbox to .01\p@{}\rlap{\reset@font\rm
        \hskip -13.4cm(\theequation)}}
\makeatother

\newcommand{\der}{\partial}
\newcommand{\Om}{\Omega}

\newcommand{\wt}{\widetilde}

\newcommand{\Si}{{\Sigma}}
\newcommand{\SO}{{\Sigma_\ast}}
\newcommand{\SC}{{\Sigma_c}}
\newcommand{\SCr}{{\Sigma_{c}}}

\newcommand{\SCAr}{{\Sigma_{c}^{A}}}
\newcommand{\SCBr}{{\Sigma_{c}^{B}}}
\newcommand{\bj}{{\bar\jmath}}
\newcommand{\bo}{{\textrm{Box}}}

\newcommand{\SCA}{{\Sigma_{c}^{A}}}
\newcommand{\SCB}{{\Sigma_{c}^{B}}}

\newtheorem{The}{Theorem}[section]
\newtheorem{Lem}[The]{Lemma}
\newtheorem{Pro}[The]{Proposition}
\newtheorem{Def}[The]{Definition}
\newtheorem{rem}[The]{Remark}
\newtheorem{cor}[The]{Corollary}

\begin{document}

\title
[Transversal submanifolds in stratified groups]
{On transversal submanifolds and their measure}
\author{Valentino Magnani}
\address{Dipartimento di Matematica, Universit\`a di Pisa, Largo Bruno Pontecorvo 5, 56127, Pisa, Italy}
\email{magnani@dm.unipi.it}
\author{Jeremy T. Tyson}
\address{Department of Mathematics, University of Illinois, 1409 West Green St., Urbana, IL 61801 USA}
\email{tyson@math.uiuc.edu}

\author{Davide Vittone}
\address{Dipartimento di Matematica, Universit\`a di Padova, via Trieste 63, 35121 Padova, Italy}
\email{vittone@math.unipd.it}

\thanks{The first author acknowledges the support of the European Project ERC AdG *GeMeThNES*.}
\thanks{The second author acknowledges the support of the US National Science Foundation Grants DMS-0901620 and DMS-1201875.}
\thanks{The third author acknowledges the support of MIUR, GNAMPA of INDAM (Italy), University of Padova, Fondazione CaRiPaRo Project ``Nonlinear Partial Differential Equations: models, analysis, and control-theoretic problems'' and University of Padova research project ``Some analytic and differential geometric aspects in Nonlinear Control Theory, with applications to Mechanics''.}
\date{\today}


\maketitle


\medskip

\begin{quote}
{\small {\bf Abstract:}
We study the class of transversal submanifolds. We characterize their blow-ups at transversal points and prove a  negligibility theorem for their ``generalized characteristic set'', with respect to the Carnot-Carath\'eodory Hausdorff measure. This set is made by all points of non-maximal degree. Observing that $C^1$ submanifolds in Carnot groups are generically transversal, the previous results prove that the ``intrinsic measure'' of $C^1$ submanifolds is generically equivalent to their Carnot-Carath\'eodory Hausdorff measure. As a result, the restriction of this Hausdorff measure to the submanifold can be replaced by a more manageable integral formula, that should be seen as a ``sub-Riemannian mass''. Another consequence of these results is an explicit formula, only depending on the embedding of the submanifold, that computes the Carnot-Carath\'eodory Hausdorff dimension of $C^1$ transversal submanifolds.}
\end{quote}

\tableofcontents


%
%
%
%
%
\section{Introduction}
%
%
%
%
%

A stratified group $\G$ is a connected, simply connected nilpotent Lie group, whose Lie algebra
$\cG$ has a special grading that allows for the existence of natural dilations along with
a homogeneous distance that respect both dilations and group operation.
The first developments of Geometric Analysis in the non-Riemannian framework of stratified groups
were mainly focused on geometric properties of domains in relation with Sobolev embeddings (see, for instance,
\cite{CDG}, \cite{FGW}, \cite{GN}), problems from the calculus of variations (e.g., \cite{BSCV07}, \cite{DGN3}),
differential geometric calculus on hypersurfaces (\cite{DGN5}), and the structure of finite perimeter sets (a very
incomplete list of references includes \cite{Amb1}, \cite{CG},  \cite{DGN1},  \cite{FSSC4}, \cite{FSSC5}, and \cite{Mag5}).
The preceding lists of references are far from exhaustive, representing only a small sample of the rapidly expanding
literature in the field of sub-Riemannian geometric analysis.

The study of finite perimeter sets and domains naturally connects with the study of hypersurfaces
and their Hausdorff measure. The fact that this measure is constructed by a fixed homogeneous distance
of the group is understood.
An important object in this context is the so-called $\G$-perimeter measure. It can be defined
using a volume measure and a smooth left invariant metric on the horizontal subbundle of the group.
This measure is equivalent, in the sense of \eqref{equivalence} below, to the $(Q-1)$-dimensional Hausdorff measure
either of the reduced boundaries in step two Carnot groups, \cite{FSSC5}, or of the topological boundaries of $C^1$
smooth domains in arbitrary stratified groups, \cite{Mag5}, where $Q$ is the Hausdorff dimension of $\G$.

The $\G$-perimeter measure for regular sets has a precise integral formula that replaces the Hausdorff measure
and that does not contain the homogeneous distance. In fact, it is more manageable for minimization problems.
In the development of Geometric Measure Theory on stratified groups, a natural question arises: what is
the ``right measure'' replacing the $\G$-perimeter measure for higher codimensional sets?

In \cite{MV}, a general integral formula for the ``intrinsic measure'' of $C^1$ submanifolds has been
found: let $\Sigma$ be a $C^1$ smooth submanifold of $\G$ and define
\begin{equation}\label{intrmeas}
\mu_\Si(U)=\int_{\Phi^{-1}(U)}\|(\der_{t_1}\Phi\wedge\cdots\wedge\der_{t_p}\Phi)_{D,\Phi(t)}\|\,dt
\end{equation}
where $\Phi:A\to U\subset\Sigma$ is a local parametrization of $\Sigma$, $A$ is an open set of $\R^p$
and $D$ is the degree of $\Sigma$, see Subsection~\ref{Subsect:Deg} for more details.
This measure yields the perimeter measure in codimension one and in several cases it is equivalent to
${\mathcal H}^D\res\Sigma$ up to geometric constants, where $D$ is the Hausdorff dimension of $\Sigma$, namely,
\begin{equation}\label{equivalence}
C^{-1}\; \cH^D\res\Si\leq \mu_\Si  \leq C\;\cH^D\res\Si\,.
\end{equation}
This equivalence already appears in \cite{MV} for $C^{1,1}$ smooth submanifolds in stratified groups,
under the key assumption that points of degree less than $D$ are ${\mathcal H}^D$-negligible.
Under this assumption, the equivalence \eqref{equivalence} is a consequence of a ``blow-up theorem''
performed at each point of degree $D$, see \cite[Theorem~1.1]{MV}.
For more details on the notion of degree, see Subsection~\ref{Subsect:Deg}.

The previously mentioned ${\mathcal H}^D$-negligibility condition holds in many cases: for $C^{1,1}$ smooth submanifolds in two step stratified groups \cite{Mag12B}, and in the Engel group \cite{LeDMag16}, for $C^1$ smooth non-horizontal submanifolds in all stratified groups \cite{Mag5,Mag12A}, and for $C^1$ smooth curves in all groups \cite{KorMag}.
In all these cases, the equivalence \eqref{equivalence} holds. In fact, when $\Sigma$ is $C^{1,1}$
this is a consequence of the blow-up theorem of \cite{MV}, while for the case of $C^1$ smoothness
the blow-up at points of degree $D$ is established in \cite{Mag12A} for non-horizontal submanifolds
and in \cite{KorMag} for all curves.

Surprisingly, for $C^1$ smooth submanifolds in stratified groups the equivalence \eqref{equivalence}
is an intriguing open question. One of the reasons behind this new difficulty is that, in higher codimension,
submanifolds may belong to different classes, namely, they may have different Hausdorff dimensions,
while keeping the same topological dimension. Simple examples of this phenomenon are given by the one dimensional homogeneous subgroups, that have different Hausdorff dimensions according to their degree. 
Clearly, analogous examples can be easily found for higher dimensional homogeneous subgroups. 
It is instructive to compare these cases with that of codimension one submanifolds,
whose Hausdorff dimension must equal $Q-1$.

Can we detect the ``right'' class of submanifolds that has the ``good behaviour'' of hypersurfaces, and replaces
them in higher codimension?  When the codimension is ``low'', precisely less than the dimension of
the {\em horizontal fibers}, this class is formed by non-horizontal submanifolds, for which \eqref{equivalence} holds,
\cite{Mag12A}. In higher codimension, this class is formed by {\em transversal submanifolds}.
A transversal submanifold is easily defined as a top-dimensional submanifold among all submanifolds
having the same topological dimension $p$. We have a precise formula for this maximal Hausdorff dimension $D(p)$, see
Section~\ref{Sect:preliminary} for precise definitions.

In this paper, we prove that transversal submanifolds in arbitrary codimension have properties
similar to those of hypersurfaces. In fact, our main result is that \eqref{equivalence} holds for all
$C^1$ smooth transversal submanifolds in arbitrary stratified groups.
This follows by combining two key results: a blow-up theorem and a negligibility result, that are stated below.
The estimates \eqref{equivalence} show in particular that the Hausdorff dimension of $C^1$ smooth transversal submanifolds
is equal to $D(p)$. This fact should be compared with \cite[0.6.B]{Gro96}, where M. Gromov provides a formula
for the Hausdorff dimension of generic submanifolds. Gromov also introduces the number $D_H(\Sigma)$ associated with
a submanifold $\Sigma$; this number coincides with the degree $d(\Sigma)$ introduced in Subsection~\ref{Subsect:Deg},
see \cite[Remark~2]{Mag12B} for a proof of this fact.

Another motivation for our study of transversal submanifolds is that $C^1$ smooth submanifolds are generically transversal, namely, ``most'' $C^1$ submanifolds are transversal. This suggests that these submanifolds are important in the subsequent study of higher codimensional submanifolds in Carnot groups. 
The fact that transversality is a generic property can be seen for instance as a simple consequence of our Lemma~\ref{goodtangentvectors} and then arguing as in \cite[Section~4]{Mag12A}. 

The main results of our work are a ``blow-up theorem'' and an $\cH^{D(p)}$-negligibility theorem
for all $C^1$ smooth transversal submanifolds. These theorems extend the blow-up theorem of \cite{Mag12A} and
the negligibility theorem of \cite{Mag5}.
%
%
%
%
\begin{The}[Blow-up theorem]\label{BlwTrsv}
If $\Sigma$ is a $C^1$ smooth submanifold and $x\in\Sigma$ is transversal, then, for every compact
neighbourhood $F$ of $0$, we have
\begin{equation}\label{eq:BlwTrsv}
F\cap\delta_{1/r}(x^{-1}\Sigma)\to F\cap\Pi_\Sigma(x)\quad\mbox{as}\quad r\to0^+
\end{equation}
where the convergence is in the sense of the Hausdorff distance between compact sets
and $\Pi_\Sigma(x)$ is a $p$-dimensional normal homogeneous subgroup of $\G$, having Hausdorff dimension equal to $D(p)$. Moreover, the following limit holds
\begin{equation}\label{eq:BlwTrsvMeas}
\lim_{r\to0^+}\frac{\tilde\mu\big(\Si\cap B(x,r)\big)}{r^{D(p)}}=
\frac{\theta^d_g\big(\big(\tau_\Sigma(x)\big)_{D(p),x}\big)}
{\|\big(\tau_{\Sigma,\tilde g}(x)\big)_{D(p),x}\|}\,.
\end{equation}
\end{The}
The proof of this theorem is given in Section~\ref{Sect:Transv}. This section, along with
Section~\ref{Sect:preliminary}, also contains the definitions of the relevant notions.
It is worth to mention that in the case of $C^{1,1}$ regularity the blow-up at transversal points
is already contained in \cite{MV}. In our case, where $\Sigma$ is only $C^1$, 
the approach of \cite{MV} does not apply, so we follow the method used in \cite{KorMag} for curves. 
The point here is to provide a special ``weighted reparametrization'' of $\Sigma$ around the blow-up point, see \eqref{etat}. Our next main result is the following generalized negligibility theorem.
\begin{The}[Negligibility theorem]\label{theo:neglig}
Let $\Si\subset\G$ be a $p$-dimensional $C^1$ submanifold and let $\SC\subset\Si$ denote the subset of
points with degree less than $D(p)$. We have
\begin{equation}\label{eqnegl}
\Haus^{D(p)}(\SC)=0\,.
\end{equation}
\end{The}
We refer to Section \ref{sec:neglig} for the definition of the {\em generalized characteristic set} $\SC$.
The proof of Theorem~\ref{theo:neglig} relies on covering arguments and a number of technical lemmata,
that aim to estimate the behaviour of the number of small balls covering the generalized characteristic set. 
The difficulty here is to properly translate the information on the lower degree of the points into concrete
estimates on the best ``local coverings'' around these points, see Lemma~\ref{lem:xtilde},
Lemma~\ref{lemma2} and Lemma~\ref{lemma3}. Theorem~\ref{theo:neglig} extends to Euclidean Lipschitz submanifolds using standard arguments, see Theorem \ref{theo:negligLip}. Arguing the same way, one also realizes that estimates \eqref{equivalence} extend to all Euclidean Lipschitz transversal submanifolds. 

The method to prove Theorem~\ref{theo:neglig} can also be used to establish new estimates on the Carnot-Carath\'eodory Hausdorff dimension of $\SC$ for $C^{1,\la}$ submanifolds in Carnot groups, where $0<\la\leq1$. These estimates, proved in Theorem~\ref{thmsizecharset}, show that the Carnot-Carath\'eodory Hausdorff dimension of $\SC$ can be estimated from above by a bound smaller than $D(p)$. Both Theorems~\ref{theo:neglig} and \ref{thmsizecharset} generalize some results proved, in the Heisenberg group framework, in the fundamental paper \cite{Bal03}, compare Remark~\ref{final-remark}. We do not know whether the estimates of Theorem \ref{thmsizecharset} are sharp. Even in Heisenberg groups, this sharpness seems to be an interesting open question. We refer to \cite{BPR} for results and open problems akin to that of estimating the size of $\SC$.

The validity of \eqref{equivalence} for a large class of submanifolds makes the intrinsic measure \eqref{intrmeas} a reasonable notion of ``sub-Riemannian mass''. This should be seen for instance in the perspective of studying special classes of isoperimetric inequalities, when either the filling current or the filling submanifold must be necessarily transversal, as it occurs for higher dimensional fillings in Heisenberg groups.

%
%
%
%
%
\section{Notation and preliminary results}\label{Sect:preliminary}
%
%
%
%
%
%
%
%
%
\subsection{Carnot groups and exponential coordinates}
Let us start with a brief introduction to stratified groups; we refer to \cite{FS} for more details on the subject.

Let $\G$ be a connected and simply connected Lie group with stratified Lie algebra
$ \cG=V_1\oplus\dots\oplus V_\iota $ of step $\iota$, satisfying the conditions
$V_{i+1}=[V_1,V_i]$ for every $i\geq 1$ and $V_{\iota+1}=\{0\}$. We set
\begin{equation}\label{m_j}
n_j:=\dim V_j\quad\text{and}\quad m_j:=n_1+\dots+n_j,\qquad j=1,\dots,\iota\,;
\end{equation}
we will also use $m_0:=0$. The degree $d_j$ of $j\in\{1,\dots,n\}$ is defined by the condition
\[
m_{d_j-1}+1\leq j\leq m_{d_j}.
\]
We denote by $n$ the dimension of $\cG$, therefore $n=m_\iota$.
We say that a basis $(X_1,\dots,X_n)$ of $\cG$ is {\em adapted to the stratification}, or in short {\em adapted}, if
\[
X_{m_{j-1}+1},\dots, X_{m_j}\text{ is a basis of $V_j$ for any $j=1,\dots, \iota$}.
\]
In the sequel, we will fix a {\em graded metric} $g$ on $\G$, namely, a left invariant Riemannian metric on $\G$ such that the subspaces $V_k$ are orthogonal.
\begin{Def}\rm
An adapted basis $(X_1,\ldots,X_n)$ of $\cG$ that is also orthonormal with respect to a
left invariant Riemannian metric is a {\em graded basis}.
\end{Def}
Clearly, the Riemannian metric in the previous definition must be necessarily graded.
When either an adapted or a graded basis is understood, we identify $\G$ with $\R^n$ by
the corresponding exponential coordinates of the first kind.

We use two different ways of denoting points $x$ of $\G$ with respect to fixed exponential coordinates of the first kind adapted to a graded basis of $\cG$.
We use both the standard notation with ``lower indices''
\[
x=(x_1,\dots, x_n)\in\R^n
\]
and the one with ``upper indices''
\[
x=(x^1,x^2,\dots,x^\iota)\in \R^{n_1}\times\R^{n_2}\times\dots\times\R^{n_\iota}\,,
\]
where clearly $x^j=(x_{m_{j-1}+1},\dots,x_{m_j})\in\R^{n_j}$ for all $j=1,\dots,\iota$.
By the Baker-Campbell-Hausdorff formula, the group law reads in coordinates as
\begin{equation}\label{BCHQ}
x\cdot y=x+y+Q(x,y)\,,
\end{equation}
for a suitable polynomial function $Q:\R^n\times\R^n\to\R^n$. Precisely, $Q=(Q_1,\dots, Q_n)$ can be written in the form
\begin{equation}\label{eq0.2}
Q_j(x,y)=\sum_{k,h:d_k<d_j,d_h<d_j}R_j^{kh}(x,y)(x_ky_h-x_hy_k)\quad\forall j=1,\dots,n
\end{equation}
for suitable polynomial functions $R_j^{kl}$. It follows that for any bounded set $K\subset\G$ there exists $C=C(K)>0$ such that
\begin{equation}\label{eq0.1}
|Q(x,y)|\leq C|x||y|\qquad\text{for any }x,y\in K\,.
\end{equation}
\subsubsection{Stratified groups as abstract vector spaces}
To emphasize some intrinsic notions on stratified groups, while preserving the ease of using a linear structure,
stratified groups can be also regarded as abstract vector spaces.
In fact, connected and simply connected nilpotent Lie groups are diffeomorphic to their Lie algebra through the exponential mapping $\exp:\cG\to\G$, that is an analytic diffeomorphism.
As a result, we equip the Lie algebra of $\G$ with a Lie group operation, given by the
Baker-Campbell-Hausdorff series, that makes this Lie group isomorphic to the original $\G$.
This allows us to consider $\G$ as an abstract linear space, equipped with a polynomial operation and
a grading $\G=H_1\oplus\cdots\oplus H_\iota$. Under this identification a graded basis
$(X_1,\ldots,X_n)$ becomes an orthonormal basis of $\G$ as a vector space, where
$(X_{m_{j-1}+1},\ldots,X_{m_j})$ is an orthonormal basis of $H_j$ for all $j=1,\ldots,\iota$.
\begin{rem}\rm
When a stratified group $\G$ is seen as an abstract vector space, equipped with a graded basis
$X_1,\ldots,X_n$, then the associated graded metric $g$ makes this basis orthonormal.
As a result, the metric $g$ becomes the Euclidean metric with respect to the corresponding coordinates $(x^1,x^2,\ldots,x^\iota)$.
\end{rem}
\subsubsection{Dilations}
For every $r>0$, a natural group automorphism $\delta_r:\cG\to\cG$ can be defined as the unique
algebra homomorphism such that
\[
\delta_r(X):=r X\qquad\text{for every }X\in V_1.
\]
This one parameter group of linear isomorphisms constitutes the family of the so-called {\em dilations} of $\cG$.
They canonically yield a one parameter group of dilations on $\G$ and can be denoted by
the same symbol. With respect to our coordinates, we have
\[
\delta_r(x_1,\dots,x_j,\dots, x_n)=(rx_1,\dots,r^{d_j}x_j,\dots, r^\iota x_n)\,.
\]
\subsubsection{Left translations}
For each element $x\in\G$, the group operation of $\G$ automatically defines the corresponding {\em left translation}
$l_x:\G\to\G$, with $l_x(z)=xz$ for all $z\in\G$. Right translations $r_x$ are defined in analogous way.

\subsection{Metric facts}
We will say that $d$ is a {\em homogeneous distance} on $\G$ if it is a continuous distance on $\G$ satisfying the following conditions
\begin{equation}\label{defhomogmetr}
d(zx,zy)=d(x,y)\quad\text{and}\quad d(\delta_r(x),\delta_r(y))=rd(x,y)\qquad\forall\:x,y,z\in\G,\:r>0.
\end{equation}
Important examples of homogeneous distances are the well known Carnot-Carath\'eo\-do\-ry distance and those constructed in \cite{FSSC5}. It is easily seen that two homogeneous distances are always equivalent. We will denote by $B(x,r)$ and $B_E(x,r)$, respectively, the open balls of center $x$ and radius $r$ with respect to a (fixed) homogeneous distance $d$ and the Euclidean distance on $\R^n\equiv \G$.

For $r>0$ we introduce the {\em boxes}
\[
\begin{split}
\bo(0,r):=\, & \{y\in \R^n:|y^j| < r^j\ \forall j=1,\dots,\iota\}\\
=\, & (-r,r)^{n_1}\times (-r^2,r^2)^{n_2}\times\dots\times (-r^\iota,r^\iota)^{n_\iota}\\
\bo(x,r):=\, & x\cdot \bo(0,r),\quad x\in\G\,.
\end{split}
\]
By homogeneity, it is easy to observe for any homogeneous distance $d$ there exists $C_{BB}=C_{BB}(d)\geq 1$ such that
\begin{equation}\label{eqstar}
\bo(x,r/C_{BB})\subset B(x,r) \subset \bo(x,C_{BB}r)\,.
\end{equation}
We will also use the notation
\[
\bo_E^\mu(0,r):=\{y\in\R^\mu:|y_j|< r\ \forall\: j=1,\dots,\mu\} = (-r,r)^\mu.
\]
When given $0<s<r$ and a linear subspace $W$ of $\R^\mu$ we pose
\[
\bo_{W^\perp\oplus W}^{\mu}(0; r, s):=\{y\in \bo_E^\mu(0,r):|\pi_W(y)|<s\}\,,
\]
where $\pi_W(y)$ is the canonical projection of $y$ on $W$. If $w_1,\dots,w_H$ is an orthonormal basis of
$W$ we clearly have
\begin{equation}\label{eqstarstar}
\{y\in \bo_E^\mu(0,r):|\langle y,w_i\rangle| < \tfrac{s}{\sqrt H}\ \forall i=1,\dots,H\}\ \subset\ \bo_{W^\perp\oplus W}^{\mu}(0; r, s)\,.
\end{equation}

From now on, a homogeneous distance $d$ is fixed.
We will use several times the following simple fact.
\begin{Lem}\label{lem:xtilde}
There exists $C=C(d)>0$ with the following property. For any fixed $r\leq 1$, $x\in B_E(0,r)$ and $j\in\{1,\dots,\iota\}$ there exists $\tilde x\in\G$ such that
\begin{align}
&\tilde x^1=\dots=\tilde x^j=0,\quad d(x,\tilde x)\leq C r^{1/j}\label{xtilde1} \\
& |\tilde x^h-x^h|\leq C r^2\quad\text{for any }h=j+1,\dots, \iota. \label{eq7bis}
\end{align}
\end{Lem}
\begin{proof}
In the case $j=1$, we define
\[
\begin{split}
\tilde x=\,&  x\cdot(-x^1,0,\dots,0)=  (x^1,\dots,x^\iota)\cdot(-x^1,0,\dots,0)\\
=\,& (0,x^2+O(r^2),\dots,x^\iota+O(r^2))\,,
\end{split}
\]
where the last equality follows from \eqref{eq0.1} and $O(\cdot)$ is understood with respect to the Euclidean norm.
By \eqref{eqstar} we have
\[
d(x,\tilde x)=d(0,(-x^1,0,\dots,0))\leq C_{BB}r
\]
whence \eqref{xtilde1} and \eqref{eq7bis} follow.

We now argue by induction on $j\geq 2$, assuming the existence of some $\bar x$ such that
\[
\bar x^1=\dots=\bar x^{j-1}=0,\quad d(x,\bar x)\leq C r^{1/(j-1)}\quad \mbox{and}\quad |\bar x^h-x^h|\leq C r^2
\]
for any $h=j,\dots, \iota$. Defining $\tilde x=\bar x\cdot (0,\dots,0,-\bar x^j,0,\dots,0)$, applying both
\eqref{eq0.2} and \eqref{eq0.1} we obtain
\[
\begin{split}
\tilde x=\,&  (\bar x^1,\dots,\bar x^\iota)\cdot(0,\dots,0,-\bar x^j,0,\dots,0) \\
=\,& (0,\dots,0,0,\bar x^{j+1}+O(r^2),\dots,\bar x^{\iota}+O(r^2))\,.
\end{split}
\]
Thus, by inductive hypothesis we get $\tilde x =(0,\dots,0,0, x^{j+1}+O(r^2),\dots, x^{\iota}+O(r^2))$.
As a result, we arrive at the following inequalities
\[
\begin{split}
d(x,\tilde x)\leq\,& d(x,\bar x) + d(\bar x,\tilde x) \leq C r^{1/(j-1)} + d(0,(0,\dots,0,-\bar x^j,0,\dots,0))\\
\leq\,&  C r^{1/j} + C_{BB}|\bar x^j|^{1/j}
= C r^{1/j} + C_{BB}|x^j+O(r^2)|^{1/j} \leq  \widetilde C r^{1/j}\,.
\end{split}
\]
that complete the proof.
\end{proof}
\subsubsection{Hausdorff measures and coverings}
For the sake of completeness, we recall the definitions of Hausdorff measures.
Let $q\geq 0$ and $\delta>0 $ be fixed; we define
\[
\begin{split}
& \Haus^q_\delta(E):=\inf\left\{\sum_{i=1}^\infty (\mathrm{diam}\:E_i)^q: E\subset \cup_i E_i,\ \mathrm{diam}\:E_i <\delta\right\}\\
& \Shaus^q_\delta(E):=\inf\left\{\sum_{i=1}^\infty (\mathrm{diam}\:B_i)^q: E\subset \cup_i B_i,B_i=B(x_i,r_i)\text{ balls},\ \mathrm{diam}\:B_i<\delta\right\}\,.
\end{split}
\]
The $q$-dimensional {\em Hausdorff measure} of $E\subset\G$ is
\[
\Haus^q(E):=\lim_{\de\to 0^+}\ \Haus^q_\delta(E)
\]
while the $q$-dimensional {\em spherical Hausdorff measure} of $E$ is
\[
\Shaus^q(E):=\lim_{\de\to 0^+}\ \Shaus^q_\delta(E)\,.
\]
The {\em Hausdorff dimension} of $E$ is
\[
\dim_H E:= \inf \{q:\Haus^q(E)=0\} = \sup\{q:\Haus^q(E)=+\infty\}\,.
\]
It is well-known that $\dim_H \G$ coincides with the {\em homogeneous dimension} $Q:=n_1+2n_2+\dots+\iota n_\iota$ of $\G$.
The standard Euclidean Hausdorff  measure on $\G=\R^n$ is denoted by $\Haus^q_{|\cdot|}$. For more information on the properties of these measures, see for instance \cite{Federer69, mattila,simon}.

We state without proof the following simple fact.
\begin{Pro}\label{prop:esponenti}
Let $\theta>0$ and let $E\subset X$, where $X$ is a metric space.
If for all $\ep\in(0,1)$ the set $E$ can be covered by $N_\ep$ balls of radius $\ep^\beta$ with $N_\ep\leq C\, \ep^{-q}$
and $C>0$ independent from $\ep$, then the Hausdorff dimension of $E$ is not greater than $q/\beta$.
\end{Pro}

The following result, see e.g. \cite[Theorem 3.3]{simon}, will be useful in the sequel.
\begin{The}[$5r$-covering]\label{5rcovering}
Let $(X,d)$ be a separable metric space and $E\subset X$; let $r>0$ be fixed. Then there exists a subset $F\subset E$ at most countable such that
\[
E\subset \bigcup_{x\in F} B(x,5r)\quad\text{and}\quad B(x,r)\cap B(x',r)=\emptyset\ \forall x,x'\in F, x\neq x'\,.
\]
\end{The}
\subsection{Multi-indices, degrees and maximal dimension}\label{multiindex}
We denote by $I_p$ the set of those multi-indices  $\al=(\al_1,\dots,\al_p)\in\{1,\dots,n\}^p$ such that $1\leq \al_1<\dots<\al_p\leq n$.
We also set
\[
d(\al):=d_{\al_1}+\dots+d_{\al_p}\,.
\]
We denote by $D(p)$ the maximum integer $d(\al)$ when $\al$ varies in $I_p$. We call this number the
{\em maximal dimension}, that is uniquely defined for any given $p\in\{1,2,\ldots,n\}$.
Clearly, $D(n)$ equals the homogeneous dimension $Q$ of $\G$.
The maximal dimension can be computed in the following way. Define $\ell=\ell(p)$ by imposing
\begin{equation}\label{defell}
\left\{\begin{array}{ll}
\ell:=\iota & \text{if }p\leq n_\iota\\
\displaystyle\sum_{j=\ell+1}^\iota n_j < p\leq \sum_{j=\ell}^\iota n_j\quad & \text{otherwise.}
\end{array}\right.
\end{equation}
Clearly, $\ell$ depends on $p$ and it can be equivalently defined by $\ell(p):=d_{n+1-p}$,
that represents the lowest possible degree among tangent vectors of $\text{span}\{X_{n-p+1},\ldots,X_n\}$,
where $(X_1,\ldots,X_n)$ is an adapted basis of the stratified Lie algebra $\cG$.
It is also easy to see that
\begin{equation}\label{eq0.3}
D(p)= \sum_{j=\ell(p)+1}^\iota j n_j +\ell(p)\, \Big(p-\sum_{j=\ell(p)+1}^\iota  n_j\Big)\,,
\end{equation}
where the two summations in \eqref{eq0.3} have to be understood as 0 when $\ell(p)=\iota$.
We also set
\begin{equation}\label{eq0.3bis}
 r_p:= p-\sum_{j=\ell(p)+1}^\iota  n_j\geq 1,
\end{equation}
so that
\begin{equation}\label{eq0.3ter}
 p=r_p+n_{\ell(p)+1}+\dots+n_\iota\quad\text{and}\quad D(p)= \ell(p) r_p + (\ell(p)+1)n_{\ell(p)+1}+\dots+\iota n_\iota\,.
\end{equation}
It is worth noticing that $D(p)=\beta_+(p)$, where $\be_+$ is the {\em upper dimension comparison function for $\G$}, introduced in \cite{BTW}.
%
%
%

%
%
%
%
%
\subsection{Degree of submanifolds, projections and subdilations}\label{Subsect:Deg}
%
%
%
%
%

By $\Si\subset\G$, we denote a $p$-dimensional Lipschitz submanifold of $\G$. We define the {\em singular set}
\[
\SO:=\{x\in\Si: C_x\Si\text{ is not a $p$-dimensional subspace of $\R^n$}\}\,,
\]
where $C_x\Si$ is the (Euclidean) tangent cone to $\Si$ at $x$, i.e.,
\[
C_x\Si:=\left\{tv\in\R^n:t\geq 0, v=\lim_{i\to\infty}\tfrac{x_i-x}{|x_i-x|}\text{ for some sequence }(x_i)_{i\in\N}\subset\Si\text{ with }x_i\to x\right\}\,.
\]
We have the following fact.

\begin{Pro}\label{prop:negligibilitySigma0}
For any $p$-dimensional Lipschitz submanifold $\Si\subset\G$, we have
\begin{equation}\label{eq1}
\Haus^{D(p)}(\SO)=0\,.
\end{equation}
\end{Pro}
\begin{proof}
Since $\Haus^p_{|\cdot|} (\SO) = 0$, by \cite[Proposition 3.1]{BTW} our claim immediately follows.
\end{proof}
Given a point $x\in\Si\setminus\SO$, we denote by $\tau_\Si(x)$ its tangent vector, i.e., the $p$-dimensional multivector associated with the $p$-plane $C_x\Si$. We can write
\[
\tau_\Si(x)=\sum_{\alpha\in I_p} c_\alpha X_\alpha\,,
\]
where $X_\al:=X_{\al_1}\wedge\ldots\wedge X_{\al_p}$. We then define the {\em degree} $d_\Si(x)$ of $\Si$ at $x$ as
\[
d_\Si(x):=\max\{d(\al):\al\in I_p\text{ and }c_\alpha\neq 0\}.
\]
The {\em degree} of $\Si$ is $d(\Si):=\max\{d_\Si(x):x\in\Si\setminus\SO\}$. Clearly, we have $d(\Si)\leq D(p)$.
The underlying metric $g$ on $\G$ gives a natural scalar product on multivectors, whose norm will be denoted by
$\|\cdot\|$. If $\tilde g$ is any Riemannian metric on $\G$, then at any $x\in\G$ we have a canonically defined
scalar product on any $\Lambda_pS_x$ where $S_x\subset T_x\G$ is a $p$-dimensional subspace.
We will denote by $\|\cdot\|_{\tilde g,x}$ its corresponding norm.
\begin{Def}\rm
Let $\tilde g$ be any Riemannian metric on $\G$, let $\Sigma$ be a $p$-dimensional Lipschitz submanifold
and let $x\in\Sigma\setminus\SO$. We define the {\em unit tangent $p$-vector} with respect to $\tilde g$
as follows
\begin{equation}
\tau_{\Sigma,\tilde g}(x)=\frac{\tau_\Sigma(x)}{\|\tau_\Sigma(x)\|_{\tilde g,x}}\,,
\end{equation}
where $\tau_\Sigma(x)$ is any tangent $p$-vector of $\Sigma$ at $x$.
\end{Def}
Dilations of $\G$ canonically extend to dilations on multivectors as follows
\[
(\Lambda_p\delta_r)(v_1\wedge\ldots\wedge v_p)=(\delta_rv_1)\wedge\ldots\wedge(\delta_rv_p)
\]
for all $v_1,\ldots,v_p\in\G$, therefore we have
\[
(\Lambda_p\delta_r)(X_\alpha)=(\Lambda_p\delta_r)(X_{\al_1}\wedge\ldots\wedge X_{\al_p})=r^{d_{\alpha_1}+\cdots+d_{\alpha_p}}\,X_\alpha
=r^{d(\alpha)}\,X_\alpha\,.
\]
\begin{Def}\rm
A $p$-vector $v\in\Lambda_p\G$ is {\em homogeneous} of degree $l\in\N\setminus\{0\}$
if $(\Lambda_p\delta_r)v=r^l v$ for all $r>0$.
\end{Def}
Our scalar product on $\Lambda_p(\G)$ allows us to introduce the following canonical projections,
hence homogeneous multivectors of different degrees are orthogonal.
\begin{Def}\rm
Let $p,D\in\N$ be such that $1\leq p\leq D\leq D(p)$. Let us introduce the linear subspace
$\Lambda_{D,p}(\G)$ of $\Lambda_p\G$ made by all homogeneous $p$-multivectors of degree $D$.
With respect to the scalar product of $\G$, the following orthogonal projection
\[
\pi_D:\Lambda_p(\G)\to\Lambda_{D,p}(\G)
\]
is uniquely defined. We say that $\pi_D$ is the {\em projection of degree $D$}.
If we consider a $p$-vector $t\in\Lambda_pS_x$ with $S_x\subset T_x\G$, a pointwise projection
$\pi_{D,x}(t)$ is automatically defined, taking the left translated multivector
$(\Lambda_pdl_{x^{-1}})(t)\in \Lambda_p(T_0\G)$, identifying $T_0\G$ with $\G$ and applying
$\pi_D$ to this translated multivector, hence
\[
 \pi_{D,x}(t)=(\Lambda_pdl_x)\circ\pi_D\circ(\Lambda_pdl_{x^{-1}})(t).
\]
To simplify notation, both projection $\pi_D$ applied to $v\in\Lambda_p\G$ and
$\pi_{D,x}$ applied to $w\in\Lambda_pS_x$ will be also denoted by $(v)_D$ and $(w)_{D,x}$,
respectively.
\end{Def}
\begin{rem}\rm
The previous notions allow us to consider the ``density function'' with respect to $\tilde g$, defined as
$\Sigma\ni x\to\|(\tau_{\Sigma,\tilde g}(x))_D\|$. This function naturally appears in the
representation of the $D$-dimensional ``intrinsic measure'' of a submanifold.
\end{rem}
The following lemma will be useful in the sequel. It can be proved repeating exactly
the same arguments of \cite[Lemma 3.1]{MV}.
\begin{Lem}\label{goodtangentvectors}
Let $\Si\subset\G$ be a $p$-dimensional Lipschitz submanifold and fix $x\in\Si\setminus\SO$.
Then we can find a graded basis $X_1,\dots,X_n$ of $\cG$ and a basis $v_1,\dots,v_p$ of $T_x\Si$ such that, writing $v_j=\sum_{i=1}^n C_{ij} X_i(x)$, we have
\begin{equation}\label{matriceC}
C:=(C_{ij})_{\substack{i=1,\dots,n\\j=1,\dots,p}}=\left[\begin{array}{c|c|c|c}
Id_{\alpha_1} & 0 & \cdots & 0\\
0 & \ast & \cdots & \ast\\
\hline 0 & Id_{\alpha_2} & \cdots & 0 \\
0 & 0 & \cdots & \ast\\
\hline \vdots & \vdots & \ddots & \vdots\\
\hline 0 & 0 & \cdots & Id_{\alpha_\iota}\\
0 & 0 & \cdots & 0
\end{array}\right]
\end{equation}
where $\alpha_k$ are integers satisfying $0\leq\alpha_k\leq n_k$ and $\alpha_1+\dots+\alpha_\iota=p$.
The symbols $0$ and $\ast$ denote null and arbitrary matrices of the proper size, respectively. We have
\begin{equation}\label{eq2}
d_\Si(x)=\sum_{k=1}^\iota k\alpha_k\,.
\end{equation}
\end{Lem}
\begin{rem}
{\upshape As already observed in \cite[Remark 3.2]{MV}, the previous lemma along with its proof are understood
to hold also in the case where some $\alpha_k$ possibly vanishes.
In this case the $\alpha_k$ columns of \eqref{matriceC} containing $Id_{\alpha_k}$ and
the corresponding vectors $v^k_j$ are meant to be absent.}
\end{rem}
The integers $\alpha_1,\ldots,\alpha_\iota$ of Lemma~\ref{goodtangentvectors} define a ``sub-grading''
for a $p$-dimensional subspace of $\R^n$, so that in analogy with the integers $m_j$ defined in \eqref{m_j}, we set
\begin{equation}\label{mu_j}
 \mu_0=0 \quad \mbox{and} \quad \mu_k=\sum_{l=1}^k\alpha_l\quad\mbox{for all}\quad k=1,\ldots,\iota.
\end{equation}
This new grading allows us to define for every $j\in\{1,\ldots,p\}$ the
{\em subdegree} $\sigma_j$ defined as follows
\begin{equation}\label{subdegree}
 \sigma_j:=k \quad \mbox{if and only if}\quad \mu_{k-1}<j\leq\mu_k\quad \mbox{for some}\quad
k\in\{1,\ldots,\iota\}\,.
\end{equation}
The corresponding {\em subdilations} $\lambda_r:\R^p\to\R^p$ are defined as follows
\[
 \lambda_r(\xi_1,\ldots,\xi_p)=(r^{\sigma_1}\xi_1,r^{\sigma_2}\xi_2,\ldots,r^{\sigma_p}\xi_p)\quad
\mbox{for all}\quad r>0.
\]
\subsection{Transversal points and transversal submanifolds}
Let us fix a $p$-dimensional Lipschitz submanifold $\Sigma$ and consider
$x\in\Si\setminus\SO$, where $\SO$ is its singular set.
We say that $x$ is {\em transversal} if $d_\Sigma(x)$ equals the maximal dimension $D(p)$.
In this case we say that $\Sigma$ is a {\em transversal submanifold}, that is equivalent to the
condition $d(\Si)=D(p)$.
\begin{rem}\rm
For hypersurfaces, transversal points coincide with noncharacteristic points and when $p\geq n-n_1$ a $p$-dimensional submanifold of $\G$ is transversal if and only if it is {\em non-horizontal}, according to the terminology of \cite{Mag12A}.
\end{rem}
The following corollary is an easy consequence of the fact that $X_i(0)=e_i$.

\begin{cor}\label{chartransv}
Under the assumptions and notations of Lemma~\ref{goodtangentvectors}, the point $x$
is transversal if and only if the following conditions hold
\begin{equation}\label{flauto}
\alpha_{\iota}=n_\iota,\;\alpha_{\iota-1}=n_{\iota-1},\;\ldots,\;\alpha_{\ell+1}=n_{\ell+1},\;
\alpha_{\ell}=r_p,\;\alpha_{\ell-1}=0,\;\ldots,\;\alpha_1=0,
\end{equation}
where $\ell=\ell(p)$ is defined by \eqref{defell} and $r_p$ is defined in \eqref{eq0.3bis}.
If $x=0$ is transversal, then the vectors $v_1,\dots,v_p$ in Lemma~\ref{goodtangentvectors} constitute the columns of the matrix
\begin{equation}\label{matriceCtransv}
C^0=(C^0_{ij})=\left[\begin{array}{c|c|c|c|c}
\ast  & \cdots & \cdots & \cdots & \ast\\
\hline
Id_{r_p}  & 0 & \cdots & \cdots & 0  \\
 0  & \ast & \cdots & \cdots & \ast\\
\hline
0   & Id_{n_{\ell+1}} & 0 & \cdots & 0\\
\hline
\vdots  & \ddots & \ddots & \cdots & \vdots\\
\hline
\vdots  & \ddots & \cdots & \ddots & 0 \\
\hline
0  & \cdots & \cdots & 0& Id_{n_\iota}
\end{array}\right]\,.
\end{equation}
\end{cor}
The previous corollary shows that at transversal points the associated grading given by the integers
of \eqref{mu_j} and \eqref{subdegree} yields
\[
\mu_0=\cdots=\mu_{\ell-1}=0,\quad\mu_\ell=r_p\quad\mbox{and}
\quad \mu_{\ell+j}=r_p+\sum_{i=1}^{j}n_{\ell+i}\quad\mbox{for all}\quad j=1,\ldots,\iota-\ell,
\]
therefore the subdegrees are the following ones
\begin{equation}\label{degsigmaj}
\sigma_1=\ell ,\ldots,\sigma_{r_p}=\ell\quad\mbox{and}\quad \sigma_{r_p+s+\sum_{i=1}^{j}n_{\ell+i}}
=\sigma_{\mu_{\ell+j}+s}=\ell+j+1
\end{equation}
for all $s=1,\ldots, n_{\ell+j+1}$ and $j=0,1,\ldots,\iota-\ell-1$, where the term $\sum_{i=1}^j n_{\ell+i}$
in the previous formulae is meant to be zero when $j=0$.

%
%
%
%
%
\section{Blow-up at transversal points}\label{Sect:Transv}
%
%
%
%
%

This section is devoted to the proof of Theorem~\ref{BlwTrsv} stated in the Introduction.
We have first to recall some more notions
and fix other auxiliary objects. First of all $\tilde g$ will denote any auxiliary Riemannian metric on $\G$.
The corresponding Riemannian surface measure induced on a $C^1$ smooth submanifold $\Sigma\subset\G$
will be denoted by $\tilde\mu$.
\begin{Def}\label{metfact}\rm
A graded metric $g$ on $\G$ is fixed and we set $\mathbb B=\{x\in\G: d(0,x)<1\}$, where
$d$ is a homogeneous distance. The stratified group $\G$ is seen as an abstract vector space
and $S$ denotes one of its $p$-dimensional linear subspaces. We consider any simple $p$-vector $\tau$
associated to $S$. Then we define the  {\em metric factor}
\begin{equation}\label{f:metfact}
 \theta_g^d(\tau)={\mathcal H}^p_{|\cdot|}(S\cap \mathbb B)\,.
\end{equation}
Here $|\cdot|$ denotes the Euclidean metric on $\G$ with respect to a fixed graded basis $(X_1,\ldots,X_n)$
and the sets $S$ and $\mathbb B$ are represented with respect to the associated coordinates of the first kind.
\end{Def}
\begin{rem}\rm
In the previous definition, any other simple $p$-vector
$\lambda\tau$ with $\lambda\neq0$ defines the same subspace $S$. Conversely, whenever a simple $p$-vector
$\zeta$ is associated to $S$, that is $\zeta=\zeta_1\wedge\cdots\wedge\zeta_p$ and $(\zeta_1,\ldots,\zeta_p)$
is basis of $S$, then $\zeta=t\,\tau$ for some $t\neq0$.
\end{rem}
\begin{rem}\rm
The metric factor only depends on the Riemannian metric $g$ and the homogeneous distance $d$.
In fact, under the assumptions of Definition~\ref{metfact}, let us consider another graded basis
$(Y_1,\ldots,Y_n)$ with associated coordinates $(y_1,\ldots,y_n)$ of the first kind.
Then the linear change of variables from these coordinates to the original coordinates
$(x_1,\ldots,x_n)$ associated to $(X_1,\ldots,X_n)$ is an isometry of $\R^n$, hence the number
\eqref{f:metfact} is preserved under the coordinates $(y_i)$.
\end{rem}
About the statement of Theorem~\ref{BlwTrsv}, we wish to clarify that the Lie subgroup $\Pi_\Si(x)$
appearing in \eqref{eq:BlwTrsv} is also homogeneous in the sense that is closed under dilations.
Furthermore, it is a $p$-dimensioanl homogeneous subgroup of $\G$ of the form
\[
Z\oplus H_{\ell+1}\oplus\cdots\oplus H_\iota,
\]
where $S\subset H_\ell$ is a linear space of dimension $r_p$.
The integers $\ell$ and $r_p$ are defined in \eqref{defell} and \eqref{eq0.3bis}. In particular, 
$\Pi_\Si(x)$ is also a normal subgroup. The same subgroup is more conveniently defined later in the proof of Theorem~\ref{BlwTrsv}, see \eqref{defPiSigma}.

\begin{proof}[Proof of Theorem~\ref{BlwTrsv}]
First of all, our claim allows us to assume that there exists an open neighbourhood $U\subset\R^p$ of the origin
such that $\Psi:U\ra\Sigma$ is a $C^1$ smooth diffeomorphism with $\Psi(0)=x$.
Defining the translated submanifold $\Sigma_x=l_{x^{-1}}(\Sigma)$ , we observe that
\[
d_{\Sigma_x}(0)=d_\Sigma(x)=D(p)=d(\Sigma)=d(\Sigma_x).
\]
We consider the translated diffeomorphism $\phi=l_{x^{-1}}\circ\Psi$, with
$\phi:U\to\Sigma_x$. Taking into account Corollary~\ref{chartransv}, we have
a graded basis of left invariant vector fields $X_1,\ldots,X_n$
and linearly independent vectors $v_1,\ldots,v_p\in T_0\Sigma_x$ such that the matrix $C^0=(C^0_{ij})$,
defined by $v_j=\sum_{j=1}^n C^0_{ij} X_i(0)=\sum_{j=1}^n C^0_{ij} e_i$ is given by \eqref{matriceCtransv}.
The vector fields $X_i$ in our coordinates have the form
\begin{equation}\label{X_i}
X_i=\sum_{l=1}^n a_i^l\,e_l
\end{equation}
and their (nonconstant) coefficients satisfy (see e.g. \cite{FS})
\begin{equation}\label{ailX}
a_i^l=\left\{\begin{array}{ll}
\delta_i^l & 	d_l\leq d_i \\
\mbox{\scriptsize homogeneous polynomial of degree $d_l-d_i$} & \mbox{otherwise,}
\end{array}\right.
\end{equation}
where homogeneity refers to intrinsic dilations of the group.
After a linear change of variable on $\phi$, we can also assume that $\der_{t_i}\phi(0)=v_i$ for all $i=1,\ldots,p$,
where each $v_i$ is the $i$-th column of \eqref{matriceCtransv}.
Let $\pi_0:\R^n\to\R^p$ be the projection
\[
\pi_0(x_1,\ldots,x_n)=(x_{m_{\ell-1}+1},\ldots,x_{m_{\ell-1}+r_p},x_{m_\ell+1},\ldots,x_n)\,.
\]
Thus, $d(\pi_0\circ\phi)(0)$ is invertible and the inverse mapping theorem provides us with new variables
$y=(y_{m_{\ell-1}+1},\ldots,y_{m_{\ell-1}+r_p},y_{m_\ell+1},\ldots,y_n)$ such that $\Sigma_x$ is represented
near the origin by $\gamma=\phi\circ(\pi_0\circ\phi)^{-1}$, which can be written as follows
\begin{equation*}
\big(\gamma_1(y),\ldots,\gamma_{m_{\ell-1}}(y),y_{m_{\ell-1}+1},\ldots,y_{m_{\ell-1}+r_p},\gamma_{m_{\ell-1}+r_p+1}(y),\ldots,
\gamma_{m_\ell}(y),y_{m_\ell+1},\ldots,y_n\big)
\end{equation*}
and it is defined in some smaller neighbourhood $(-c_1,c_1)^p\subset U$ for some $c_1>0$.
Furthermore, since $d(\pi_0\circ\phi)(0)$ is the identity mapping of $\R^p$, we get
\begin{equation}\label{derv_j}
(\der_1\gamma)(0)=v_1, \quad (\der_2\gamma)(0)=v_2, \quad\ldots\quad (\der_p\gamma)(0)=v_p,
\end{equation}
so that we have continuous functions $C_{ij}(y)$ with $C_{ij}(0)=C^0_{ij}$ such that
\begin{equation}\label{derjgamma}
(\der_j\gamma)(y)=\sum_{i=1}^n C_{ij}(y) X_i(\gamma(y))\quad \mbox{for all}\quad j=1,\ldots,p\,.
\end{equation}
Due to the structure of $C^0$ given in \eqref{matriceCtransv}, whenever $\sigma_j=\ell$, or equivalently
when $j=1,\ldots,r_p$, we have
\begin{equation}\label{strucC}
C_{ij}(y)= \delta_{i-m_{\ell-1},j}+o(1)\quad  \mbox{for any $i$ such that $d_i\geq\ell$}\,.
\end{equation}
When $\ell<\sigma_j\leq\iota$, or equivalently in the case $j>r_p$ and $j=\mu_{\sigma_j-1}+1,\ldots,\mu_{\sigma_j}$,
we get
\begin{equation}
\quad C_{ij}(y)=
\left\{\begin{array}{ll}
\delta_{ i-m_{\sigma_j-1},j-\mu_{\sigma_j-1} }+o(1) & \mbox{if $d_i=\sigma_j$ or $1\leq i-m_{\sigma_j-1}\leq n_{\sigma_j}$} \\
o(1) & \mbox{if $d_i>\sigma_j$}\,.
\end{array}\right.
\end{equation}
Let us introduce the $C^1$ smooth homeomorphism $\eta:\R^p\to\R^p$ as follows
\begin{equation}\label{etat}
\eta(t)=\bigg(\frac{|t_1|^{\sigma_1}}{\sigma_1}\sgn(t_1),\ldots,\frac{|t_p|^{\sigma_p}}{\sigma_p}\sgn(t_p)\bigg)\,,
\end{equation}
where its inverse mapping is given by the formula
\[
\zeta(\tau)=\bigg(\sgn(\tau_1)\sqrt[\sigma_1]{\sigma_1|\tau_1|},\ldots,\sgn(\tau_p)\sqrt[\sigma_p]{\sigma_p|\tau_p|}\bigg)
\]
and all $\sigma_j$ satisfy \eqref{degsigmaj}.
We consider the $C^1$ smooth reparametrization $\Gamma(t)=\gamma\big(\eta(t)\big)$ with partial derivatives
\begin{equation}\label{partjG}
\der_{t_j}\Gamma(t)=|t_j|^{\sigma_j-1}\,(\der_j\gamma)(\eta(t))=
|t_j|^{\sigma_j-1}\sum_{s,i=1}^n C_{ij}(\eta(t))a_i^s(\Gamma(t))\,e_s
\end{equation}
for all $j=1,\ldots,p$, where we have used both \eqref{X_i} and \eqref{derjgamma}.
We first observe that
\begin{equation}
 \Gamma_i(t)=o(|t|^{d_i})\quad\mbox{for} \quad \ 1\leq d_i<\ell.
\end{equation}
In fact, we have $\gamma(0)=0$ and $\eta(t)=O(|t|^\ell)$, hence
\begin{equation}\label{gilessml}
 \Gamma_i(t)=\gamma_i(\eta(t))=O(|\eta(t)|)=
  \left\{\begin{array}{ll}
 o(|t|^{d_i}) & \mbox{if $d_i<\ell$} \\
 O(|t|^\ell) &  \mbox{if $d_i=\ell$}\,.
 \end{array}\right.
\end{equation}
The main point is to prove the following rates of convergence
\begin{equation}\label{induct0}
\left\{\begin{array}{ll}
 \Gamma_i(t)=O(|t|^{\ell}) & \mbox{for}\;\ m_{\ell-1}<i\leq m_{\ell-1}+r_p \\
\Gamma_i(t)=o(|t|^{\ell}) & \mbox{for}\;\  m_{\ell-1}+r_p<i\leq m_\ell\,.
\end{array}\right.
\end{equation}
Since the first equation of \eqref{induct0} is already contained in \eqref{gilessml}, we have nothing to prove in the case
$r_p=n_\ell$ (because \eqref{induct0} does not have the second case).
Thus, we will assume that $r_p<n_\ell$ and then prove the second formula of \eqref{induct0}.
First of all, we apply \eqref{partjG} and compute the following partial derivatives
\begin{equation}\label{partjGamma}
\der_{t_j}\Gamma_i(t)=|t_j|^{\sigma_j-1}\,(\der_j\gamma_i)(\eta(t))=
|t_j|^{\sigma_j-1}\,\sum_{k=1}^n C_{kj}(\eta(t))a_k^i(\Gamma(t))
\end{equation}
for all $j=1,\ldots,p$ and all $i=1,\ldots,n$.
If $1\leq j\leq r_p$, we rewrite the previous sum as
\begin{equation*}
|t_j|^{\sigma_j-1}\bigg(\sum_{k: d_k<\ell} C_{kj}(\eta(t))a_k^i(\Gamma(t))+
\sum_{k: d_k\geq\ell} C_{kj}(\eta(t))a_k^i(\Gamma(t))\bigg)\,.
\end{equation*}
As a consequence, taking into account \eqref{strucC}, it follows that
\begin{equation*}
\begin{split}
\der_{t_j}\Gamma_i(t)&=|t_j|^{\ell-1}\bigg(a_{j+m_{\ell-1}}^i(\Gamma(t))+\sum_{k: d_k<\ell} C_{kj}(\eta(t))a_k^i(\Gamma(t))
+\sum_{k: d_k\geq\ell} o(1)\,a_k^i(\Gamma(t))\bigg)\\
&=|t_j|^{\ell-1}\bigg(a_{j+m_{\ell-1}}^i(\Gamma(t))+o(1)+\sum_{k: d_k<\ell} C_{kj}(\eta(t))a_k^i(\Gamma(t)) \bigg)\,.
\end{split}
\end{equation*}
Since $d_k<\ell$, we have that $a^i_k$ is a nonconstant homogenous polynomial.
It follows that $a^i_k\circ\Gamma=o(1)$ and we get
\begin{equation}
\der_{t_j}\Gamma_i(t)=|t_j|^{\ell-1}\Big(a_{j+m_{\ell-1}}^i(\Gamma(t))+o(1)\Big)
\end{equation}
Since $m_{\ell-1}+j\leq m_{\ell-1}+r_p<i\leq m_\ell$ and $d_{m_{\ell-1}+j}=d_i$,
formula \eqref{ailX} implies that $a_{j+m_{\ell-1}}^i$ is the null polynomial.
It follows that
\begin{equation}\label{derjrp}
\der_{t_j}\Gamma_i(t)=o(|t|^{\ell-1})\quad\mbox{whenever}\quad 1\leq j\leq r_p\;\ \mbox{and}\;\ m_{\ell-1}+r_p<i\leq m_\ell\,.
\end{equation}
If $r_p<j\leq p$, then \eqref{partjGamma} implies that
\begin{equation*}
\der_{t_j}\Gamma_i(t)=|t_j|^{\sigma_j-1}\,\sum_{k=1}^n C_{kj}(\eta(t))a_k^i(\Gamma(t))=|t_j|^{\sigma_j-1}O(1)=O(|t|^{\sigma_j-1})\,.
\end{equation*}
Since in this case $\sigma_j>\ell$, we get in particular that
\begin{equation}\label{derjrpm}
\der_{t_j}\Gamma_i(t)=o(|t|^{\ell-1})\quad \mbox{whenever}\quad r_p<j\leq p\;\ \mbox{and}\;\ 1\leq i\leq n\,.
\end{equation}
Joining \eqref{derjrp} with \eqref{derjrpm}, it follows that
\[
\nabla\Gamma_i(t)=o(|t|^{\ell -1}) \quad\mbox{for all }\quad i=m_{\ell-1}+r_p+1,\ldots,m_\ell\,,
\]
that proves the second equation of \eqref{induct0}.

Now, we write explicitly the form of $\Gamma$ as the composition $\gamma\circ\eta$.
By the previous formulae for $\gamma$ and $\eta$, we get
\begin{equation}\label{Gammat}
\begin{split}
\Gamma(t)=\Big(\Gamma_1(t),\ldots,\Gamma_{m_{\ell-1}}&(t),\frac{|t_1|^{\ell}}{\ell}\sgn(t_1),\ldots,\frac{|t_{r_p}|^{\ell}}{\ell}\sgn(t_{r_p}),\Gamma_{m_{\ell-1}+r_p+1}(t),\ldots\\
&\ldots,\Gamma_{m_\ell}(t),\frac{|t_{r_p+1}|^{\ell+1}}{\ell+1}\sgn(t_{r_p+1}),\ldots,
\frac{|t_p|^\iota}{\iota}\sgn(t_p)\Big)\,.
\end{split}
\end{equation}
The new parametrization $\gamma$ of $\Sigma_x$ around the origin yields
$\Phi:(-c_1,c_1)^p\to\Sigma$ defined as $\Phi:=l_x\circ\gamma$, that is our ``adapted parametrization'' of
$\Sigma$ around $x$. Taking $r>0$ sufficiently small, we have
\begin{equation}
\frac{\tilde\mu(\Sigma\cap B(x,r))}{r^{D(p)}}=r^{-D(p)}
\int_{\Phi^{-1}(B(x,r))}\|(\der_{y_1}\Phi\wedge\cdots\wedge\der_{y_p}\Phi)(y) \|_{\tilde g}\,dy \,.
\end{equation}
where $\tilde\mu$ is the Riemannian surface measure induced by $\tilde g$
on $\Sigma$. We perform the change of variable $y=\lambda_rt$, where $\lambda_r$ is the
subdilation of the form
\begin{equation}
  \lambda_r(t_1,\ldots,t_p)=(r^\ell t_1,\ldots,r^\ell t_{r_p},r^{\ell+1}t_{r_p+1},\ldots,r^{\ell+1}t_{r_p+n_{\ell+1}},
\ldots,r^{\iota}t_p)
\end{equation}
that yields the formula
\begin{equation}
\frac{\tilde\mu(\Sigma\cap B(x,r))}{r^{D(p)}}=
\int_{\lambda_{1/r}(\Phi^{-1}(B(x,r)))}\|\der_{y_1}\Phi(\lambda_rt)\wedge\cdots\wedge
\der_{y_p}\Phi(\lambda_rt)\|_{\tilde g}\,dt \,.
\end{equation}
The point is then to study the ``behaviour'' of the set $\lambda_{1/r}\big(\Phi^{-1}(B(x,r))\big)$
as $r\to0^+$. To do this, we will use the formula \eqref{Gammat} for $\Gamma$ and the rates of convergence
\eqref{induct0}, taking into account the change of variables \eqref{etat}.
Since $\Phi^{-1}(B(x,r))=\gamma^{-1}(B(0,r))$, it follows that
\begin{equation}
\lambda_{1/r}\big(\Phi^{-1}(B(x,r))\big)=\left\{t\in\R^p:\delta_{1/r}\big(\gamma(\lambda_rt)\big)\in\mathbb B\right\},
\end{equation}
where $\mathbb B=\{z\in\G:d(z,0)<1\}$. We observe that
\[
 \gamma(\lambda_rt)=\Gamma(\zeta(\lambda_rt))=\Gamma(r\,\zeta(t))\,,
\]
therefore the previous rescaled set can be written as follows
\begin{equation}\label{lambda_r}
 \lambda_{1/r}\big(\Phi^{-1}(B(x,r))\big)=\left\{t\in\R^p:\delta_{1/r}\big(\Gamma(r\zeta(t))\big)\in\mathbb B\right\}\,.
\end{equation}
By \eqref{Gammat}, an element $t\in\R^p$ of the previous set is characterized by the property that
\begin{equation}\label{Gammaresc}
\begin{split}
\Big(\frac{\Gamma_1(r\zeta(t))}{r},\ldots,\frac{\Gamma_{m_{\ell-1}}(r\zeta(t))}{r^{\ell-1}},t_1,&
\ldots,t_{r_p},\frac{\Gamma_{m_{\ell-1}+r_p+1}(r\zeta(t))}{r^{\ell}},\ldots\\
&\ldots,\frac{\Gamma_{m_\ell}(r\zeta(t))}{r^\ell},t_{r_p+1},\ldots,t_p\Big)
\end{split}
\end{equation}
belongs to $\mathbb B$. This is a simple consequence of the equalities $\eta(r\zeta(t))=\lambda_r\eta(\zeta(t))=\lambda_rt$.
We now use both \eqref{gilessml} and \eqref{induct0} to conclude that the element represented in \eqref{Gammaresc}
converges to
\begin{equation}\label{Gammaresc0}
\Big(0,\ldots,0,t_1,\ldots,t_{r_p},0,\ldots,0,t_{r_p+1},\ldots,t_p\Big)
\end{equation}
as $r\to0^+$, uniformly with respect to $t$ that varies in a bounded set.
By standard facts on Hausdorff convergence, the previous limit implies the convergence in \eqref{eq:BlwTrsv} where
\begin{equation}\label{defPiSigma}
\Pi_\Si(x):=\{z\in\G:z_1=\dots=z_{m_{\ell-1}}=z_{m_{\ell-1}+r_p+1}=\dots=z_{m_\ell}=0\}\,.
\end{equation}
It can be easily seen that $\Pi_\Si(x)$ is the homogeneous subgroup of $\G$ associated with the Lie subalgebra
\[
\text{span}\:\{X_{m_{\ell-1}+1},X_{m_{\ell-1}+2},\dots, X_{m_{\ell-1}+r_p},X_{m_{\ell}+1},\dots X_n\}\,.
\]
Moreover, the convergence of the element represented in \eqref{Gammaresc} to \eqref{Gammaresc0} also gives
\begin{equation}\label{limmur0}
\lim_{r\to0^+}\frac{\tilde\mu(\Sigma\cap B(x,r))}{r^{D(p)}}=
{\mathcal H}^p_{|\cdot|}(\B\cap S)\,\|(\der_{y_1}\Phi\wedge\cdots\wedge\der_{y_p}\Phi)(0) \|_{\tilde g,x}
\end{equation}
where $S=\{(0,\ldots,0,t_1,\ldots,t_{r_p},0,\ldots,0,t_{r_p+1},\ldots,t_p)\in\R^n:t_1,\ldots,t_p\in\R\}$
and the metric unit ball $\B$ is represented with respect to the same coordinates.
Taking into account \eqref{derv_j} and the matrix \eqref{matriceCtransv}, we have the projection
\[
\pi_{D(p),0}\big((\der_{y_1}\gamma\wedge\cdots\wedge\der_{y_p}\gamma)(0)\big)=
\big(X_{m_{\ell-1}+1}\wedge\cdots\wedge X_{m_{\ell-1}+r_p}\wedge X_{m_\ell+1}\wedge\cdots\wedge X_n\big)(0)
\]
and the formulae $\der_{y_j}\Phi(x)=dl_x\big(\der_{y_j}\gamma(0)\big)$ yield
\[
\pi_{D(p),x}\big((\der_{y_1}\Phi\wedge\cdots\wedge\der_{y_p}\Phi)(0)\big)=
\big(X_{m_{\ell-1}+1}\wedge\cdots\wedge X_{m_{\ell-1}+r_p}\wedge X_{m_\ell+1}\wedge\cdots\wedge X_n\big)(x)\,.
\]
We have the unit tangent $p$-vector
\[
\tau_{\Sigma,\tilde g}(x)=\frac{\!\!\!\!\!(\der_{y_1}\Phi\wedge\cdots\wedge\der_{y_p}\Phi)(0)}
{\|(\der_{y_1}\Phi\wedge\cdots\wedge\der_{y_p}\Phi)(0) \|_{\tilde g,x} }
\]
then the previous equations for projections give
\[
\big\|\big(\tau_{\Sigma,\tilde g}(x)\big)_{D(p),x}\big\|=\frac{1}{\|(\der_{y_1}\Phi\wedge\cdots\wedge\der_{y_p}\Phi)(0) \|_{\tilde g,x} }\,.
\]
As a result, in view of Definition~\ref{metfact}, the limit \eqref{limmur0} proves our last claim \eqref{eq:BlwTrsvMeas}.
\end{proof}

%
%
%
%
%
\section{Negligibility of lower degree points in transversal submanifolds}\label{sec:neglig}
%
%
%
%
%

The aim of this section is to prove Theorem \ref{theo:neglig} for a $C^1$ $p$-dimensional transversal submanifold $\Sigma\subset\G$, where we define
\begin{equation}\label{def:SC}
\SC:=\{x\in\Si:d_\Si(x)<D(p)\}\,.
\end{equation}
Since $\Si$ is transversal, the subset $\SC$ plays the role of a {\em generalized characteristic set} of $\Si$.
Since any left translation is a diffeomorphism, for each point $x\in\Si$ there holds
\begin{equation}\label{eq4}
T_0(x^{-1}\cdot\Si)= dl_{x^{-1}}(T_x\Si)\,.
\end{equation}
Clearly, a basis for $T_0(x^{-1}\cdot\Si)$ is given by
\[
dl_{x^{-1}}(v_1),\dots,dl_{x^{-1}}(v_p)\,,
\]
where the vectors $v_1,\dots, v_p$ are given by Lemma \ref{goodtangentvectors}.
If $v_j=\sum_i C_{ij}X_i(x)$, by the left invariance of $X_i$, we have
\begin{equation}\label{eq4.1}
dl_{x^{-1}}(v_j)=\sum_{i=1}^n C_{ij}X_i(0)\quad\text{for any }j=1,\dots, p\,.
\end{equation}
In particular, $d_{x^{-1}\cdot\Si}(0)=d_\Si(x)$ and $0\in (x^{-1}\cdot\Si)_c$ if and only if $x\in\SC$.

%
Taking into account~\eqref{flauto}, we observe that $\SCr$ can be written as the disjoint union
\begin{equation}\label{unionAB}
\SCr=\SCAr\cup\SCBr\,,
\end{equation}
where we have defined
\begin{equation}\label{defSCABr}
\begin{split}
& \SCAr:=\{x\in \SCr:\exists\; \bj\geq \ell+1\text{ such that }\al_\bj<n_\bj\}\\
& \SCBr:=\{x\in \SCr:\al_j=n_j\ \forall j\geq\ell+1\text{ and }\al_\ell<r_p\}\,.
\end{split}
\end{equation}
The integer $\ell$, depending on $p$, is introduced in \eqref{defell} and the nonnegative integers
$\alpha_1,\ldots,\alpha_\iota$ are defined in Lemma~\ref{goodtangentvectors}. In particular, $\ell$ will be used throughout this section.
We notice that in the case $\ell=1$, we must have $\alpha_\ell=r_p$, hence $\SCBr=\emptyset$.

We begin by making the further assumption that $\Sigma$ is of class $C^1$ and such that $\Si\subset\phi([0,1]^p)$
for some $C^1$-regular map $\phi:[0,1]^p\to\G$. By the uniform differentiability of $\phi$, the boundedness of $\Si$
and the continuity of left translations, the following statement holds: for any $\ep>0$, there exists $\bar r_\ep>0$ such that
\begin{equation}\label{eq5}
|\langle y,w\rangle|\leq \ep r\qquad
\begin{array}{l}
\forall r\in(0,\bar r_\ep),\\
\forall x\in\Si,\ \forall y\in (x^{-1}\cdot\Si) \cap B_E(0,r),\\
\forall w\in \left( T_0(x^{-1}\cdot\Si) \right)^\perp,\ |w|=1\,,
\end{array}
\end{equation}
where $\langle\cdot,\cdot\rangle$ denotes the Euclidean scalar product. The orthogonal space $\left( T_0(x^{-1}\cdot\Si) \right)^\perp$ is understood with respect to the same product. Notice that such coordinates are associated with the basis $X_1,\dots,X_n$ given by Lemma \ref{goodtangentvectors}; in particular, they depend on the chosen basepoint $x\in\Si$.

%
%
%
%
The proof of the negligibility stated in Theorem~\ref{theo:neglig}
stems from the following key lemmata.
The proofs of these lemmata could be rather simplified; however, we present them in a form
which will be helpful for some refinement provided in Subsection~\ref{sec:Hausdorffdim}.
\begin{Lem}\label{lemma2}
Let $\Sigma$ be a $C^1$ submanifold such that $\Sigma\subset\f([0,1]^p)$ for some $C^1$ map $\f:[0,1]^p\to\G$; let $\theta:=1/\ell$. Then, there exists a constant $C_A=C_A(\Si)>0$ such that the following property holds. For any $x,\ep, r$ satisfying
\begin{equation}\label{eq:condlemma2}
x\in\SCAr,\quad \ep\in(0,1)\quad\text{and}\quad 0<r\leq\min\{\bar r_\epsilon,\ep^\ell\},
\end{equation}
the set $(x^{-1}\cdot\Si) \cap B_E(0,r)$ can be covered by a family $\{B_i:i\in I\}$ of CC balls with radius $r^{\theta}$ such that
\[
\# I\leq C_A\: \epsilon\: r^{p-\theta D(p)}\,.
\]
\end{Lem}
\begin{proof}
From now on, the numbers $C_i$, with $i=1,2,\dots$, will denote positive constants depending only on $\Si,p,\G$ and the fixed homogeneous distance $d$. For the reader's convenience, we divide the proof into several steps.

{\em Step 1.} By Theorem~\ref{5rcovering}, we get a countable family $\{B(x_i,r^{\theta}):i\in I\}$ such that
\begin{equation}\label{covering}
\left\{\begin{array}{l}
x_i\in (x^{-1}\cdot\Si) \cap B_E(0,r)\\
(x^{-1}\cdot\Si) \cap B_E(0,r) \subset \textstyle\bigcup_{i\in I} B(x_i,r^{\theta})\\
B(x_i, r^{\theta}/5)\cap B(x_h,r^{\theta}/5) = \emptyset\;\mbox{when}\; i\neq h.
\end{array}\right.
\end{equation}
We have to estimate $\# I$. By Lemma \ref{lem:xtilde}, for any $i\in I$ there exists $\tilde x_i$ such that
\begin{equation}\label{eq7}
\begin{split}
&\tilde x_i^1=\dots=\tilde x_i^\ell=0,\quad d(x_i,\tilde x_i)\leq C r^{1/\ell} = C r^{\theta},  \\
&|\tilde x^h_i-x^h_i|\leq C r^2\quad\text{for any }h=\ell+1,\dots,\iota.
\end{split}
\end{equation}
Therefore, taking into account \eqref{eqstar}, we achieve
\begin{equation}\label{eq777}
B(x_i,r^{\theta})\subset B(\tilde x_i,(1+C)r^{\theta}) \subset \bo(\tilde x_i,C_1r^{\theta})\,.
\end{equation}
Let us also point out that both \eqref{eq7} and the fact that $x_i\in B_E(0,r)$ give
\begin{equation}\label{eq9bis}
|\tilde x_i| \leq C_2 r\,.
\end{equation}

{\em Step 2.} Let us prove that there exists $C_3>0$ such that, for any $i\in I$, there holds
\begin{equation}\label{eq8}
\bo(\tilde x_i, C_1r^{\theta})\subset\Om,
\end{equation}
where we have set
\[
\Om := (-C_3r^{\theta},C_3 r^{\theta})^{n_1}\times(-C_3 r^{2\theta},C_3 r^{2\theta})^{n_2}\times\dots\times (-C_3 r^{\ell\theta},C_3 r^{\ell\theta})^{n_\ell}\times \bo_E^{\mu}(0, C_3 r)
\]
and $\mu:=n-m_{\ell}$. To this aim we fix $y\in \bo(0,C_1 r^\theta)$, that is
\begin{equation}\label{eq8.1}
|y^j| < (C_1 r^\theta)^j\quad\forall j=1,\dots,\iota\,,
\end{equation}
and prove that $\tilde x_i\cdot y\in\Om$. By explicit computation
\begin{equation}\label{eq9}
\begin{split}
\tilde x_i\cdot y\, &= (0,\dots,0,\tilde x_i^{\ell+1},\dots,\tilde x_i^\iota)\cdot(y^1,\dots,y^\iota)\\
&= (y^1,\dots,y^\ell,\tilde x_i^{\ell+1}+y^{\ell+1},\tilde x_i^{\ell+2}+y^{\ell+2}+O(r^{1+\theta}),\dots,\tilde x_i^\iota+y^{\iota}+O(r^{1+\theta}))
\end{split}
\end{equation}
where we have used
\begin{itemize}
\item \eqref{eq0.2} for the coordinates in the layers $1,\dots, \ell+1$;
\item \eqref{eq0.1} for the coordinates in the layers $\ell+2,\dots,\iota$, together with \eqref{eq9bis} and the fact that $|y|=O(r^\theta)$.
\end{itemize}
Here and in the sequel, all the quantities $O(\cdot)$ are uniform. From \eqref{eq9} and \eqref{eq8.1} it follows immediately that $\tilde x_i\cdot y\in\Om$, and \eqref{eq8} follows.

{\em Step 3.} We have not used the fact that $x\in\SCAr$ yet. By definition, there exists $\bj\geq \ell+1$ such that $\al_\bj < n_\bj$. We can also assume that $\bj$ is maximum, i.e., that $\al_j=n_j$ for any $j>\bj$; set
\[
\nu:=n_\bj+n_{\bj+1}+\dots+n_\iota = n_\bj+\al_{\bj+1}+\dots+\al_\iota\,.
\]
The last $\nu$ rows of the matrix $C$ given by Lemma \ref{goodtangentvectors} constitute a $\nu\times p$ matrix $M$ of the form
\[
M=\left[\begin{array}{c|c|c|c|c|c|c}
0 & \cdots & 0 & Id_{\alpha_\bj} & 0 & \cdots & 0\\
0 & \cdots & 0 & 0 & \ast & \cdots & \ast\\
\hline 0 & \cdots & 0 &  0 & Id_{n_{\bj+1}} & \cdots & 0 \\
\hline \vdots & \ddots & \vdots &\vdots & \vdots & \ddots & \vdots\\
\hline 0 & \cdots & 0 &  0 & 0 & \cdots & Id_{n_\iota}\\
\end{array}\right]
=
\left[\begin{array}{c|c|c|c|c}
0 & \cdots & 0 & Id_{\alpha_\bj} & 0 \\
0 & \cdots & 0 & 0 & \ast \\
\hline &\vphantom{\vdots}&&&\\
0 & \cdots & 0 &  0 & Id_{n_{\bj+1}+\dots+n_\iota}\\ 
 &&&&
\end{array}\right]\,.
\]
Since $M$ has only $\al_\bj+n_{\bj+1}+\dots+n_\iota < \nu$ nonzero columns, there exists a vector $z\in\R^\nu$ such that $|z|=1$ and $z$ is orthogonal to any of the columns of $M$. Therefore, the vector $w:=(0,z)\in \R^n\equiv\R^{n-\nu}\times\R^\nu$ is orthogonal to any of the columns of $C$. By \eqref{eq4} and \eqref{eq4.1},
taking into account that $X_k(0)=\der_{x_k}$, these columns generate $T_0(x^{-1}\cdot\Si)$. As a result,
since $\bj>\ell$, we are lead to the validity of the following conditions
\begin{equation}\label{wprop}
\left\{\begin{array}{l}
w\in (T_0(x^{-1}\cdot\Si))^\perp\\
\!|w|=1\\
w_1=w_2=\dots=w_{m_\ell}=0
\end{array}\right.\,.
\end{equation}
{\em Step 4.}
To refine the inclusion \eqref{eq8}, we will use the properties \eqref{wprop}.
By \eqref{eq5} one has $|\langle x_i, w\rangle|\leq \ep r \text{ for any }i\in I$. Define $w':=(w_{m_{\ell}+1},w_{m_{\ell}+2},\dots,w_n)\in\R^\mu$, where $\mu=n-m_\ell>\nu$ is the same number of Step 2. By \eqref{eq9} and \eqref{wprop}, for any $y\in \bo(0,C_1 r^\theta)$ we have
\[
\begin{split}
|\langle \tilde x_i\cdot y,w\rangle|\, &=  \big|\big\langle (y^1,\dots,y^\ell,\tilde x_i^{\ell+1}+y^{\ell+1},\tilde x_i^{\ell+2}+y^{\ell+2}+O(r^{1+\theta}),\dots,\tilde x_i^\iota+y^{\iota}+O(r^{1+\theta})),\\
& \hphantom{= \big|\big\langle\ }(0,\dots,0,w^{\ell+1},\dots,w^\iota) \big\rangle\big|\\
& \leq|\langle (\tilde x_i^{\ell+1},\dots,\tilde x_i^\iota) , w' \rangle| +
|\langle (y^{\ell+1},\dots,y^\iota) , w'\rangle| + O(r^{1+\theta})\\
& = |\langle (x_i^{\ell+1},\dots,x_i^\iota) , w' \rangle| + O(r^{(\ell+1)\theta}) + O(r^{1+\theta})\\
&\leq \ep r + O(r^{1+\theta})\,,
\end{split}
\]
where the second equality is justified by \eqref{eq7} and \eqref{eq8.1} and the last inequality follows from $(\ell+1)\theta= 1+\theta$.
Since all the previous $O(\cdot)$s are uniform with respect to the index $i$, we get
\[
|\langle (\tilde x_i\cdot y)_\mu,w'\rangle| \leq \ep r + C_4 r^{1+\theta}\leq (1+C_4)\ep r\,,
\]
where $(\tilde x_i\cdot y)_\mu$ is the vector made by the last $\mu$ coordinates of $(\tilde x_i\cdot y)_\mu$ and we used the fact that, by \eqref{eq:condlemma2}, $r^\theta=r^{1/\ell}\leq\ep$. Thus, by \eqref{eqstarstar} and \eqref{eq8} we obtain that
$ \bo(\tilde x_i, C_1 r^\theta)\subset\wt\Om$, where we have set
\[
\begin{split}
\wt\Om :=  (-C_3r^{\theta},C_3 r^{\theta})^{n_1}  &  \times(-C_3 r^{2\theta},C_3 r^{2\theta})^{n_2}\times\cdots\\
& \times (-C_3 r^{\ell\theta},C_3 r^{\ell\theta})^{n_\ell}\times \bo_{w^{\prime\perp}\oplus\text{span }w'}^{\mu}(0; C_3 r, C_5\ep r)\,.
\end{split}
\]
As a consequence, by \eqref{eq777} we get
$B(x_i,r^\theta/5)\subset \bo(\tilde x_i, C_1 r^\theta)\subset\wt\Om$
for all $i\in I$.

{\em Step 5.} We are ready to estimate $\# I$. The volume of $\wt\Om$ is equal to
\[
a = C_6\, \ep\, r^{\theta(n_1+2n_2\dots+\ell n_\ell)+\mu}
= C_6\, \ep\, r^{\theta(n_1+2n_2\dots+\ell n_\ell)+n_{\ell+1}+\dots+ n_\iota}\,,
\]
while each $B(x_i,r^\theta/5)$ has volume $b= C_7\, r^{\theta(n_1+2n_2+\dots+\iota n_\iota)}$.
Taking into account that the CC balls $B(x_i,r^\theta/5)$ are pairwise disjoint and contained in $\wt\Om$, we have
\[
\begin{split}
\# I  \leq \tfrac a b \stackrel{\hphantom{\eqref{eq0.3ter}}}{=} &\tfrac{C_6}{C_7} \:\ep\, r^{n_{\ell+1}+\dots+n_\iota - \theta \left((\ell+1)n_{\ell+1}+\dots+\iota n_\iota\right)}
\stackrel{\eqref{eq0.3ter}}{=}  \tfrac{C_6}{C_7} \:\ep\, r^{p-r_p -\theta(D(p)-\ell r_p)}\\
\stackrel{\hphantom{\eqref{eq0.3ter}}}{=}&\tfrac{C_6}{C_7} \:\ep\, r^{p-\theta D(p) + (\theta\ell-1)r_p}
\end{split}
\]
which proves the claim and concludes the proof of the lemma.
\end{proof}

While more subtle at certain points, the proof of Lemma~\ref{lemma3} follows the same lines of the previous one. For the reader's benefit, we will try to make the analogies between the two proofs as evident as possible.
\begin{Lem}\label{lemma3}
Under the assumptions of Lemma~\ref{lemma2} and $\ell\geq 2$, there exists $C_B=C_B(\Si)>0$ such that the following property holds. For any $x,\ep,\theta,r$ satisfying
\begin{equation}\label{eq:condlemma3}
x\in\SCBr,\quad \ep\in(0,1),\quad \tfrac{1}{\ell}<\theta\leq\tfrac{1}{\ell-1} \quad \text{and}\quad 0<r\leq\min\{\bar r_\ep, \ep^{1/(\ell\theta-1)}\},
\end{equation}
the set $(x^{-1}\cdot\Si) \cap B_E(0,r)$ can be covered by a family $\{B_i:i\in I\}$ of CC balls with radius $r^{\theta}$ such that
\[
\# I\leq C_{B} \:\ep^H\: r^{p-\theta D(p)-(\ell\theta-1)(n_\ell-r_p)},
\]
where $H=H(x):= n_\ell-\al_\ell$ and the integers $\al_j=\al_j(x)$ are those given by Lemma~\ref{goodtangentvectors}.
\end{Lem}
\begin{proof}
%
We follow the same convention of Lemma~\ref{lemma2} about the constants $C_i$.

{\em Step 1.} By the $5r$-covering theorem we can cover $(x^{-1}\cdot\Si) \cap B_E(0,r)$ by a family of CC balls $\{B(x_i,r^{\theta}):i\in I\}$ such that \eqref{covering} holds. We have once more to estimate $\# I$.
By Lemma~\ref{lem:xtilde}, for any $i\in I$ there exists $\tilde x_i$ such that
\begin{equation} \label{eq14}
\begin{array}{l}
\tilde x_i^1=\dots=\tilde x_i^{\ell-1}=0,\quad d(x_i,\tilde x_i)\leq C r^{1/(\ell-1)} \leq C r^\theta\quad
\mbox{and}   \\
|\tilde x^h_i-x^h_i|\leq C r^2\quad\text{for any }h=\ell,\dots,\iota.
\end{array}
\end{equation}
Therefore $B(x_i,r^{\theta}/5)\subset B(x_i,r^{\theta})\subset B(\tilde x_i,(1+C)r^{\theta}) \subset \bo(\tilde x_i,C_{8}r^{\theta})$.
Again
\begin{equation}\label{eq9bislemma3}
|\tilde x_i| \leq C_2 r\,.
\end{equation}

{\em Step 2.} Let us prove that there exists $C_{9}>0$ such that, for any $i\in I$, there holds
\begin{equation}\label{eq15-ep}
\bo(\tilde x_i, C_{8}r^{\theta})\subset\Om,
\end{equation}
where now
\begin{equation}\label{eq15}
\begin{split}
\Om := (-C_{9}r^{\theta},C_{9} r^{\theta})^{n_1}  &  \times(-C_{9} r^{2\theta},C_{9} r^{2\theta})^{n_2}\times\cdots\\
& \times (-C_{9} r^{(\ell-1)\theta}, C_{9} r^{(\ell-1)\theta})^{n_{\ell-1}}\times \bo_E^{\mu}(0, C_{9} r)
\end{split}
\end{equation}
and $\mu:=n-m_{\ell-1}=n_{\ell}+\dots+n_\iota$. As before we fix $y\in \bo(0,C_{8} r^\theta)$,
\begin{equation}\label{eq15.1}
|y^j| < (C_{8} r^\theta)^j\quad\forall j=1,\dots,\iota
\end{equation}
and prove that $\tilde x_i\cdot y\in\Om$. Reasoning as in Step 2 in the proof of Lemma~\ref{lemma2} we get
\begin{equation}\label{eq16}
\begin{split}
\tilde x_i\cdot y\, &= (0,\dots,0,\tilde x_i^{\ell},\dots,\tilde x_i^\iota)\cdot(y^1,\dots,y^\iota)\\
&= (y^1,\dots,y^{\ell-1},\tilde x_i^{\ell}+y^{\ell},\tilde x_i^{\ell+1}+y^{\ell+1}+O(r^{1+\theta}),\dots,\tilde x_i^\iota+y^{\iota}+O(r^{1+\theta}))
\end{split}
\end{equation}
where we have used \eqref{eq0.2}, \eqref{eq0.1}, \eqref{eq9bislemma3} and the fact that $|y|=O(r^\theta)$.
All the quantities $O(\cdot)$ are uniform.
The inclusion \eqref{eq15-ep} follows from \eqref{eq9bislemma3}, \eqref{eq16} and the fact that
\[
|y^j|< (C_{8} r^\theta)^j = C_{8}^j r^{j\theta} \leq C_{8}^j r^{j/\ell}\leq C_{8}^j r     \quad\forall j=\ell,\dots,\iota\,.
\]

{\em Step 3.} Since $x\in\SCBr$ we have by definition
\[
\al_\ell<r_p\quad\text{and}\quad \al_j=n_j\ \forall j\geq\ell+1\,.
\]
Therefore the last $\mu$ rows of the matrix $C$ from Lemma~\ref{goodtangentvectors} constitute a $\mu\times p$ matrix $M$ of the form
\[
M=\left[\begin{array}{c|c|c|c|c|c|c}
0 & \cdots & 0 & Id_{\alpha_\ell} & 0 & \cdots & 0\\
0 & \cdots & 0 & 0 & \ast & \cdots & \ast\\
\hline 0 & \cdots & 0 &  0 & Id_{n_{\ell+1}} & \cdots & 0 \\
\hline \vdots & \cdots & \vdots &\vdots & \vdots & \ddots & \vdots\\
\hline 0 & \cdots & 0 &  0 & 0 & \cdots & Id_{n_\iota}\\
\end{array}\right]
=
\left[\begin{array}{c|c|c|c|c}
0 & \cdots & 0 & Id_{\alpha_\ell} & 0 \\
0 & \cdots & 0 & 0 & \ast \\
\hline &\vphantom{\vdots}&&&\\
0 & \cdots & 0 &  0 & Id_{n_{\ell+1}+\dots+n_\iota}\\ 
 &&&&
\end{array}\right]
\]
There are $\al_\ell+n_{\ell+1}+\dots+n_\iota$ nonzero columns of $M$; therefore, the columns of $M$ span a vector subspace of $\R^\mu$ of dimension at most $\al_\ell+n_{\ell+1}+\dots+n_\iota$. Since
\[
\mu-(\al_\ell+n_{\ell+1}+\dots+n_\iota)= n_\ell-\al_\ell =H ,
\]
it follows that there exist $H$ linearly independent vectors $z_1,\dots,z_H\in \R^\mu$ such that $|z_k|=1$ and $z_k$ is orthogonal to any of the columns of $M$ for any $k=1,\dots, H$. In particular, the unit vectors
\[
w_k:=(0,z_k)\in \R^n\equiv\R^{n-\mu}\times\R^\mu,\qquad k=1,\dots, H
\]
are orthogonal to any of the columns of $C$, which form a basis of $T_0(x^{-1}\cdot\Si)$. Setting $W:=\text{span}(w_1,\dots,w_H)$ we have
\[
W\subset T_0(x^{-1}\cdot\Si)^\perp\quad\text{and}\quad \dim W=H\geq 1\,;
\]
moreover, any vector $w\in W$ is of the form
\begin{equation}\label{wprime}
w=(0,\dots,0,w^\ell,\dots,w^\iota)=(0,w')\in \R^{m_{\ell-1}}\times\R^{\mu}\,.
\end{equation}

{\em Step 4.} Again we want to refine the inclusion \eqref{eq15-ep}. By \eqref{eq5} there holds
\[
|\langle x_i, w\rangle|\leq \ep r\qquad \forall\; i\in I,\ \forall\; w\in W\text{ with }|w|=1\,.
\]
Recalling \eqref{eq16} and writing $w=(0,w')\in\R^{m_{\ell-1}}\times\R^{\mu}$ as in \eqref{wprime}, for any $y\in \bo(0,C_{8} r^\theta)$ we have
\[
\begin{split}
|\langle \tilde x_i\cdot y,w\rangle|\, &=  \big|\big\langle (y^1,\dots,y^{\ell-1},\tilde x_i^{\ell}+y^{\ell},\tilde x_i^{\ell+1}+ y^{\ell+1} +O(r^{1+\theta}),\dots,\tilde x_i^\iota+y^{\iota}+O(r^{1+\theta})),\\
\,& \hphantom{=  \big|\big\langle \ }(0,\dots,0,w^{\ell},\dots,w^\iota) \big\rangle\big|\\
& \leq|\langle (\tilde x_i^{\ell},\dots,\tilde x_i^\iota) , w' \rangle| +
|\langle (y^{\ell},\dots,y^\iota) , w'\rangle| + O(r^{1+\theta})\\
& = |\langle (x_i^{\ell},\dots,x_i^\iota) , w' \rangle| + O(r^{\ell\theta})+ O(r^{1+\theta})\\
&\leq \ep r + O(r^{\ell\theta})+ O(r^{1+\theta})\qquad\forall w\in W,|w|=1
\end{split}
\]
where we used \eqref{eq14} and \eqref{eq15.1}. Since
\[
\ell\theta = (\ell-1)\theta + \theta \leq 1+\theta,
\]
we have $r^{1+\theta}\leq r^{\ell\theta}$ and thus, since all the $O(\cdot)$s are uniform,
\[
|\langle \tilde x_i\cdot y,w\rangle| \leq \ep r + C_{10}r^{\ell\theta} \leq \max\{\ep,C_{10}r^{\ell\theta-1}\}r \leq  C_{11}\,\ep\, r \qquad\forall w\in W,|w|=1,
\]
the last inequality following from \eqref{eq:condlemma3}. Using \eqref{eqstarstar} we can then refine \eqref{eq15} to obtain
\[
B(x_i,r^\theta/5) \subset  \bo(\tilde x_i, C_{8} r^\theta)\subset\wt\Om\qquad\forall i\in I
\]
where
\[
\begin{split}
\wt\Om :=(-C_{9}r^{\theta},C_{9} r^{\theta})^{n_1} &   \times(-C_{9} r^{2\theta},C_{9} r^{2\theta})^{n_2}\times\cdots\\
& \times (-C_{9} r^{(\ell-1)\theta}, C_{9} r^{(\ell-1)\theta})^{n_{\ell-1}}\times \bo_{W^\perp\oplus W}^{\mu}(0; C_{9} r,C_{11}\ep r)\,.
\end{split}
\]

{\em Step 5.} We can now estimate $\# I$. Since $\dim W=H$, the volume of $\wt\Om$ is
\[
a\, = C_{12}\, \ep^H\, r^{\theta(n_1+2n_2\dots+(\ell-1) n_{\ell-1})+\mu}= C_{12}\, \ep^H\, r^{\theta(n_1+2n_2\dots+(\ell-1) n_{\ell-1})+n_{\ell}+\dots+n_\iota}\,,
\]
while each ball $B(x_i,r^\theta/5)$ has volume $b= C_{7} r^{\theta(n_1+2n_2+\dots+\iota n_\iota)}$.
Since the CC balls $B(x_i,r^\theta/5)$ are pairwise disjoint and contained in $\wt\Om$, we have
\begin{eqnarray*}
\# I  \leq \tfrac a b & \stackrel{\hphantom{\eqref{eq0.3ter}}}{=}& \tfrac{C_{12}}{C_{7}}\, \ep^H\,   r^{n_{\ell}+\dots+n_\iota - \theta(\ell n_\ell+\dots+\iota n_\iota)}\\
& \stackrel{{\eqref{eq0.3ter}}}{=} & \tfrac{C_{12}}{C_{7}} \, \ep^H\,  r^{n_\ell+p-r_p -\theta[\ell(n_\ell-r_p)+\ell r_p +(\ell+1)n_{\ell+1}+\dots+\iota n_\iota]}\\
& \stackrel{{\eqref{eq0.3ter}}}{=} & \tfrac{C_{12}}{C_{7}} \, \ep^H\,  r^{p+(n_\ell-r_p) -\theta[\ell(n_\ell-r_p)+D(p)]}\\
& \stackrel{\hphantom{\eqref{eq0.3ter}}}{=}& \tfrac{C_{12}}{C_{7}} \, \ep^H\,  r^{p-\theta D(p) -(\ell\theta-1)(n_\ell-r_p)}\,,
\end{eqnarray*}
as claimed.
\end{proof}

\begin{Lem}\label{intermediatelemma}
Let $\Sigma$ be a $C^1$ submanifold such that $\Si\subset\f([0,1]^p)$ for a $C^1$ map $\f:[0,1]^p\to\G$. Then $\Haus^{D(p)}(\SC)=0$.
\end{Lem}
\begin{proof}
Clearly, it will be enough to show that
\begin{equation}\label{eq:neglAB}
\Haus^{D(p)}(\SCA)=0\quad\text{and}\quad \Haus^{D(p)}(\SCB)=0.
\end{equation}

{\em Step 1.} We start by proving the first equality in \eqref{eq:neglAB}; let us follow the same convention of Lemmata~\ref{lemma2} and~\ref{lemma3} about the constants $C_i$.

Let $\ep\in (0,1)$ and $r\in(0,\min\{\bar r_\epsilon,\ep^\ell\}]$ be fixed. Since $(x,y)\to x^{-1}y$ is locally Lipschitz and $\phi$ is Lipschitz, both with respect to the Euclidean distance, we obtain $C_{13}>0$ such that
\begin{equation}\label{eq6}
\text{if $z_1,z_2\in[0,1]^p$ and $|z_1-z_2|\leq C_{13} r,\quad$ then }|\f(z_1)^{-1}\cdot\f(z_2)| < r\,.
\end{equation}
Let us divide $[0,1]^p$, in a standard fashion, into a family of closed subcubes of diameter not greater than $C_{13}r$; in this way there will be less than $C_{14} r^{-p}$ such subcubes. Let $(Q_j)_{j\in J}$ be the family of those subcubes with the property that
\[
\f(Q_j)\cap \SCAr\neq \emptyset
\]
and fix $x_j\in\f(Q_j)\cap\SCAr$. By \eqref{eq6} we have
\[
x_j^{-1}\cdot \f(Q_j) \subset (x^{-1}_j\cdot\Si) \cap B_E(0,r)\,.
\]
Writing $\theta:=1/\ell$, Lemma \ref{lemma2} ensures that $x_j^{-1}\cdot \f(Q_j)$ can be covered by (at most) $C_A\ep r^{p-\theta D(p)}$ balls of radius $r^{\theta}$; by left invariance, the same holds for $\f(Q_j)$. In particular, since $\SCAr \subset \cup_{j\in J}\f(Q_j)$ and $\# J\leq C_{14} r^{-p}$, we have that, for any $r\in(0,\min\{\bar r_\epsilon,\ep^\ell\}]$, the set $\SCAr$ can be covered by a family of CC balls with radius $r^{\theta}$ of cardinality controlled by $C_AC_{14} \ep r^{-\theta D(p)}$. Therefore
\[
\Haus^{D(p)}_{2r^{1/\ell}}(\SCAr)\leq C_AC_{14}\ep r^{-\theta{D(p)}} (2r^{\theta})^{D(p)}= 2^{D(p)}C_AC_{14}\ep
\]
whence, letting $r\to0^+$,
\[
\Haus^{D(p)}(\SCAr)\leq 2^{D(p)}C_AC_{14}\ep\,.
\]
The first part of \eqref{eq:neglAB} follows by the arbitrarity of $\ep$.

{\em Step 2.} Let us prove the second equality in \eqref{eq:neglAB}. Let $\ep\in(0,1)$ and $r\in(0, \min\{\bar r_\ep, \ep^{\ell-1}\} ]$ be fixed; we have $\ep=r^\la$ for a suitable $\la=\la(r)\in(0,\tfrac{1}{\ell-1}]$. Define $\theta=\theta(r):=\tfrac{1+\la}{\ell}$ and observe that $1/\ell<\theta\leq 1/(\ell-1)$.
As a result, we have
\begin{equation}\label{eqepeta222}
r^{\ell\theta-1}=r^\la=\ep\,;
\end{equation}
in particular, $\ep^{1/(\ell\theta-1)}=r\leq \bar r_\ep$ and  the conditions in \eqref{eq:condlemma3} are satisfied. As before, we divide $[0,1]^p$ into a family of (at most) $C_{14} r^{-p}$ closed subcubes of diameter not greater than $C_{13}r$. Let $(Q_k)_{k\in K}$ be the family of those subcubes with the property that
\[
\f(Q_k)\cap \SCBr\neq \emptyset
\]
and fix $x_k\in\f(Q_k)\cap\SCAr$. By \eqref{eq6} we have again
\[
x_k^{-1}\cdot \f(Q_k) \subset (x^{-1}_k\cdot\Si) \cap B_E(0,r)
\]
so that, by Lemma \ref{lemma3}, $x_k^{-1}\cdot \f(Q_k)$ can be covered by no more than
\[
C_B\ep^{H(x_k)} r^{p-\theta D(p)-(\ell\theta-1)(n_\ell-r_p)}
\]
balls of radius $r^{\theta}=\ep^{1/\ell}r^{1/\ell}$; by left invariance, the same holds for $\f(Q_k)$. Notice that
\[
H(x_k)= n_\ell-\al_\ell(x_k)\geq n_\ell-r_p+1\qquad\forall k\in K\,,
\]
i.e., $\f(Q_k)$ can be covered by (at most) $C_B\ep^{n_\ell-r_p+1} r^{p-\theta D(p)-(\ell\theta-1)(n_\ell-r_p)}$ balls of radius $r^\theta$. As before, this implies that
\[
\# K \leq C_BC_{14} \ep^{n_\ell-r_p+1} r^{-\theta D(p)-(\ell\theta-1)(n_\ell-r_p)}
\]
whence, using \eqref{eqepeta222},
\begin{equation}\label{lupo}
\begin{split}
\Haus^{D(p)}_{2r^{\theta}}(\SCBr)\leq & C_BC_{14}\ep^{n_\ell-r_p+1} r^{-\theta D(p)}r^{-(\ell\theta-1)(n_\ell-r_p)} (2r^{\theta})^{D(p)}\\
= & 2^{D(p)}C_BC_{14}\ep^{n_\ell-r_p+1} \ep^{-(n_\ell-r_p)}\\
= & 2^{D(p)}C_BC_{14}\ep\,.
\end{split}
\end{equation}
Observing that
\[
\lim_{r\to 0^+} r^{\theta} = \lim_{r\to 0^+} r^{1/\ell}r^{\la(r)/\ell} = \lim_{r\to 0^+} r^{1/\ell}\ep^{1/\ell}=0
\]
we can let $r\to0^+$ in \eqref{lupo} to obtain
\[
\Haus^{D(p)}(\SCBr)\leq 2^{D(p)}C_BC_{14}\ep\,.
\]
This proves the second equality in \eqref{eq:neglAB} and completes the proof.
\end{proof}

\begin{rem}\label{rem:conseqclaims}{\rm
We point out for future references the following two facts proved, respectively, in Step 1 and Step 2 of the proof of Lemma \ref{intermediatelemma}. Let $\ep\in(0,1)$ be fixed and assume that $\Sigma$ is a $C^1$ submanifold such that $\Si\subset\f([0,1]^p)$ for a $C^1$ map $\f:[0,1]^p\to\G$; then
\begin{equation}\label{conseqclaim1}
\begin{tabular}{p{13cm} }
for any $r\in(0,\min\{\bar r_\epsilon,\ep^\ell\}]$, the set $\SCAr$ can be covered by a family of CC balls with radius $r^{1/\ell}$ of cardinality at most $C_AC_{14} \ep r^{-D(p)/\ell}$.
\end{tabular}
\end{equation}
and
\begin{equation}\label{conseqclaim2}
\begin{tabular}{p{13cm} }
for any $r\in(0, \min\{\bar r_\ep, \ep^{\ell-1}\} ]$, the set $\SCBr$ can be covered by a family of CC balls, with radius $\ep^{1/\ell}r^{1/\ell}$, of cardinality at most $C_BC_{14}\ep\,\ep^{-D(p)/\ell}r^{-D(p)/\ell}$.
\end{tabular}
\end{equation}
In \eqref{conseqclaim2}, we used the fact that the cardinality of the involved family is controlled by
\[
\begin{split}
&C_BC_{14} \ep^{n_\ell-r_p+1}\, r^{-\theta D(p)-(\ell\theta-1)(n_\ell-r_p)}\\
= & C_BC_{14} \ep^{n_\ell-r_p+1}\, \ep^{-D(p)/\ell}\, r^{-D(p)/\ell}\,\ep^{-(n_\ell-r_p)}\\
=& C_BC_{14} \ep\, \ep^{-D(p)/\ell}\,  r^{-D(p)/\ell}\,,
\end{split}
\]
where we also utilized the equalities $r^\theta=\ep^{1/\ell}r^{1/\ell}$ and $r^{\ell\theta-1}=\ep$.
}\end{rem}

The proof of Theorem \ref{theo:neglig} is now at hand.

\begin{proof}[Proof of Theorem \ref{theo:neglig}]
The theorem is an easy consequence of Lemma \ref{intermediatelemma} and a standard localization argument.
\end{proof}

%
%
%
%
Actually, Theorem \ref{theo:neglig} can be generalized to Lipschitz $p$-dimensional submanifolds; recall that the singular set $\SO$ was defined at the beginning of Section \ref{Subsect:Deg}.
Clearly, the definition of $\SC$ given at \eqref{def:SC} for $C^1$ submanifolds extends to Lipschitz submanifolds
considering the subset $\Sigma\setminus\SO$ of regular points, since the pointwise degree is defined
by the existence of the pointwise tangent space.
\begin{The}\label{theo:negligLip}
Let $\Si\subset\G$ be a $p$-dimensional Lipschitz submanifold, let  $\SO$ be its singular set and 
denote by $\SC$ be the subset of points in $\Si\setminus\SO$ whose degree is less than $D(p)$. 
It follows that
\begin{equation}\label{eqneglLip}
\Haus^{D(p)}(\SO\cup\SC)=0\,.
\end{equation}
\end{The}
\begin{proof}
By definition, $\Si$ is locally the graph of a Euclidean Lipschitz function, hence without loss of generality, we can assume that $\Si\subset\f(A)$, where $\f$ is the graph function given by a Lipschitz function $f:A\to V$,
$A\subset W$ is a bounded open set of $W$ and $\G$ is seens as by $W\times V$, where $W$ and $V$ are linear subspaces
of dimensions $p$ and $n-p$, respectively.
Let $\ep>0$ be arbitrarily fixed such that $0<\ep<{\mathcal L}^p(A)$. 
By the classical Whitney's extension theorem, there exists a $C^1$ function $f_\ep:A\to V$ such that the set
\begin{equation}\label{eqbla3}
E_\ep:=\{z\in A: f_\ep(z)=f(z)\text{ and }\nabla f_\ep(z)=\nabla f(z)\}
\end{equation}
satisfies $\mathcal L^p(A\setminus E_\ep)<\ep$. 
The graph function $\f_\ep$ associated to $f_\ep$ defines the $C^1$ submanifold $\Si^\ep:=\f_\ep(A)$,
hence Theorem~\ref{theo:neglig} implies that that its generalized characteristic set $\Si_c^\ep:=\{x\in\Si^\ep:d_{\Si^\ep}(x)<D(p)\}$ is $\Haus^{D(p)}$-negligible. By the conditions of 
\eqref{eqbla3}, we have the inclusion $\SC\cap \f(E_\ep)\subset \Si_c^\ep$, hence 
$\SC\cap \f(E_\ep)$ is also $\Haus^{D(p)}$-negligible.
As a consequence of \cite[Proposition 3.1]{BTW}, there exists a geometric constant $C>0$,
only depending on the diameter of $\f(A)$ and on $\G$, such that
\[
\Haus^{D(p)}(\SC)=\Haus^{D(p)}(\SC\setminus\f(E_\ep))\leq C\,\Haus_{|\cdot|}^{p}(\SC\setminus \f(E_\ep))
\leq C\,L^p\,\ep\,,
\]
where $L>0$ is the Euclidean Lipschitz constant of $\f$. The arbitrary choice of $\ep$ implies that
$\Haus^{D(p)}(\SC)=0$ and using \eqref{eq1}, the proof is accomplished.
\end{proof}

%
%
%
%
%
\section{Size of the characteristic set for \texorpdfstring{$C^{1,\la}$}{smoother} submanifolds}\label{sec:Hausdorffdim}
%
%
%
%
%

In this section we assume that $\Si$ is a submanifold of class $C^{1,\la}$ for some $\la\in (0,1]$.
Our aim is to refine Theorem \ref{theo:neglig} and obtain estimates on the Hausdorff dimension of the characteristic set $\SC$. We first assume that $\Si\subset\f([0,1]^p)$ for some  map $\f\in C^{1,\la}([0,1]^p,\G)$. Under this assumption, there exists $C=C(\Si)>0$ such that
\begin{equation}\label{eq5refined}
|\langle y,w\rangle|\leq Cr^{1+\la}\qquad
\begin{array}{l}
\forall x\in\Si,\ \forall y\in (x^{-1}\cdot\Si) \cap B_E(0,r),\\
\forall w\in \left( T_0(x^{-1}\cdot\Si) \right)^\perp,\ |w|=1\,.
\end{array}
\end{equation}
In other words, the number $\bar r_\ep$ defined by \eqref{eq5} can be chosen to be $\bar r_\ep=(\ep/C)^{1/\la}$.

As in \eqref{unionAB}, we write $\SC=\SCA\cup\SCB$ where, following \eqref{defSCABr}, we define
\[
\begin{split}
& \SCA=\{x\in \SC:\exists\; \bj\geq \ell+1\text{ such that }\al_\bj<n_\bj\}\\
& \SCB=\{x\in \SC:\al_j=n_j\ \forall j\geq\ell+1\text{ and }\al_\ell<r_p\}\,.
\end{split}
\]
Again, if $\ell=1$, then $\SCB=\emptyset$.

\begin{Lem}\label{lemsizeA}
Let $\Sigma\subset\G$ be a $C^{1,\la}$ submanifold such that $\Sigma\subset\f([0,1]^p)$  for some map $\f\in C^{1,\la}([0,1]^p,\G)$. Then
\begin{equation}\label{eq:sizeA}
\begin{array}{ll}
\dim_H \SCA \leq D(p)-1 & \text{if }\la\geq 1/\ell\\
\dim_H \SCA \leq D(p)-\ell\la & \text{if }\la\leq 1/\ell\,.
\end{array}
\end{equation}
\end{Lem}
\begin{proof}
If $\la> 1/\ell$ we have
\[
\min\{\bar r_\ep,\ep^\ell\} = \min\{(\ep/C)^{1/\la},\ep^\ell\}=\ep^\ell
\]
for any $\ep>0$ small enough. We are then allowed to use \eqref{conseqclaim1} with $r:=\ep^\ell$ and obtain that, for any $\ep>0$ small enough, the set $\SCA$ can be covered by a family of balls with radius $\ep$ of cardinality at most $C_AC_{14}\ep\,\ep^{-D(p)}$. By Proposition \ref{prop:esponenti} we get
\[
\dim_H\SCA \leq D(p)-1\,.
\]

On the other hand, if $\la\leq 1/\ell$ we have
\[
\min\{\bar r_\ep,\ep^\ell\} = \min\{(\ep/C)^{1/\la},\ep^\ell\}=C_{15}\ep^{1/\la}
\]
for any $\ep>0$ small enough; we have utilized the usual convention on constants $C_i$. Using \eqref{conseqclaim1} with $r:=C_{15}\ep^{1/\la}$, we get that, for any $\ep>0$ small enough, the set $\SCA$ can be covered by a family of balls with radius $r^{1/\ell}=C_{16}\ep^{1/(\ell\la)}$ of cardinality at most $C_{17}\ep\,\ep^{-D(p)/(\ell\la)}$. By Proposition \ref{prop:esponenti} we get
\[
\dim_H\SCA \leq D(p)-\ell\la
\]
and this concludes the proof.
\end{proof}

\begin{Lem}\label{lemsizeB}
Let $\Sigma\subset\G$ be a $C^{1,\la}$ submanifold such that $\Sigma\subset\f([0,1]^p)$  for some map $\f\in C^{1,\la}([0,1]^p,\G)$; assume $\ell\geq 2$. Then
\begin{equation}\label{eq:sizeB}
\begin{array}{ll}
\dim_H \SCB \leq D(p)-1 & \text{if }\la\geq \frac{1}{\ell-1}\vspace{.1cm}\\
\dim_H \SCB \leq D(p)-\frac{\ell\la}{1+\la} & \text{if }\la\leq \frac{1}{\ell-1}\,.
\end{array}
\end{equation}
\end{Lem}
\begin{proof}
If $\la> 1/(\ell-1)$ we have
\[
\min\{\bar r_\ep,\ep^{\ell-1}\} = \min\{(\ep/C)^{1/\la},\ep^{\ell-1}\}=\ep^{\ell-1}
\]
for any $\ep>0$ small enough. We are then allowed to use \eqref{conseqclaim2} with $r:=\ep^{\ell-1}$ and obtain that, for any $\ep>0$ small enough, the set $\SCB$ can be covered by a family of balls with radius $\ep^{1/\ell}\ep^{(\ell-1)/\ell}=\ep$ of cardinality at most
\[
C_BC_{14}\,\ep\,\ep^{-D(p)/\ell}r^{-D(p)/\ell} = C_BC_{14}\, \ep^{-D(p)+1}.
\]
By Proposition \ref{prop:esponenti} we get
\[
\dim_H\SCB \leq D(p)-1\,.
\]

On the other hand, if $\la\leq 1/(\ell-1)$ we have
\[
\min\{\bar r_\ep,\ep^{\ell-1}\} = \min\{(\ep/C)^{1/\la},\ep^{\ell-1}\}=C_{18}\ep^{1/\la}
\]
for any $\ep>0$ small enough. Using \eqref{conseqclaim2} with $r:=C_{18}\ep^{1/\la}$, we get that, for any $\ep>0$ small enough, the set $\SCB$ can be covered by a family of balls with radius $\ep^{1/\ell}r^{1/\ell}=C_{17}\ep^{(\la+1)/(\ell\la)}$ of cardinality at most
\[
C_BC_{14}\ep\,\ep^{-D(p)/\ell}\,r^{-D(p)/\ell} = C_BC_{14}\ep\,\ep^{-\frac{D(p)}{\ell}}\,\ep^{-\frac{D(p)}{\ell\la}} = C_BC_{14}\ep^{-\frac{\la+1}{\ell\la}D(p)+1}\,.
\]
By Proposition \ref{prop:esponenti} we get
\[
\dim_H\SCB \leq D(p)-\tfrac{\ell\la}{\la+1}
\]
and this concludes the proof.
\end{proof}

Recalling that $\SCB=\emptyset$ if $\ell=1$, Lemmata \ref{lemsizeA} and \ref{lemsizeB} immediately lead to the following result.

\begin{The}\label{thmsizecharset}
Let $\Si$ be a $p$-dimensional submanifold of $\G$ of class $C^{1,\la}$, $\la\in(0,1]$.
It follows that
\begin{eqnarray}\label{dim_HH}
&&\left\{\begin{array}{ll}
\dim_H \SC \leq D(p)-\la & \text{if }\ell=\ell(p)=1\vspace{.1cm}\\
\dim_H \SC \leq D(p)-1 & \text{if }\ell\geq 2\text{ and }\la\geq \tfrac{1}{\ell-1}\vspace{.1cm}\\
\dim_H \SC \leq D(p)-\tfrac{\ell\la}{1+\la}\quad & \text{if }\ell\geq 2\text{ and }\la\leq \tfrac{1}{\ell-1}
\end{array}\right.\,.
\end{eqnarray}
\end{The}
\begin{rem}\label{final-remark}{\rm
It is interesting to analyze Theorem \ref{thmsizecharset} when the Carnot group $\G$ is the Heisenberg group $\H^n$. 
In this case, $\ell=\ell(p)=1$ for all $p=2,\dots,2n$ and Theorem \ref{thmsizecharset} reads as
\begin{equation}\label{Heisenberg}
\dim_H \SC \leq p+1-\la
\end{equation}
for any $p$-dimensional submanifold $\Si\subset\H^n$ of class $C^{1,\la}$. These estimates coincide with the results stated in Remark~1, page 72 of \cite{Bal03}. In the special case $p=1$, we have $\ell=2$, hence
Theorem~\ref{thmsizecharset} gives
\begin{equation}\label{Heisenbergp=1}
\dim_H \SC \leq D(1)-\frac{2\la}{1+\la}=\frac{2}{1+\lambda}\leq 2-\la\,,
\end{equation}
where the last inequality is strict for all $\la\in(0,1)$ and $2-\la=p+1-\la$.
Thus, in this special case of curves ($p=1$), the estimates \eqref{dim_HH} improve that of 
Remark~1, page 72 of \cite{Bal03}. 
}\end{rem}

\end{document}